%% file: main.tex
\documentclass[12pt,a4paper,twoside,reqno]{amsart}

\usepackage{charter}
\usepackage{dutchcal}
\usepackage[OT2,T1]{fontenc} % for sha
\DeclareSymbolFont{cyrletters}{OT2}{wncyr}{m}{n}
\DeclareMathSymbol{\Sha}{\mathalpha}{cyrletters}{"58}

\usepackage{ragged2e}
\usepackage[top=0.9in, bottom=1.15in, left=1.1in, right=1.1in]{geometry}

\usepackage{amsmath,amsthm,amssymb,enumerate,amsfonts,mathrsfs,amsxtra,amscd,latexsym,graphicx}

\usepackage{hyperref} 
\usepackage{cleveref}
\usepackage{longtable}
\usepackage{lscape}
\usepackage{multicol}
\usepackage{array}
\newcolumntype{L}[1]{>{\raggedright\let\newline\\\arraybackslash\hspace{0pt}}m{#1}}  %%%change this if tables are warped
\newcolumntype{R}[1]{>{\raggedleft\let\newline\\\arraybackslash\hspace{0pt}}m{#1}}

\theoremstyle{definition}
\newtheorem{thm}{theorem}[section]
\newtheorem{lemma}[thm]{Lemma}
\newtheorem{theorem}[thm]{Theorem}
\newtheorem{proposition}[thm]{Proposition}

\newtheorem{definition}[thm]{Definition}
\newtheorem{conjecture}[thm]{Conjecture}

\theoremstyle{plain}
\newtheorem{example}{Example}

\numberwithin{equation}{section}
\numberwithin{table}{section}

\title{Elliptic Curves and Thompson's Sporadic Simple Group}

\author{Maryam Khaqan}
\address{Department of Mathematics, Emory University, 400 Dowman Drive, Atlanta, GA 30322}
\email{maryam.khaqan@emory.edu}


\input{PreambleSymbolShorthands}

\usepackage[authoryear]{natbib}
\setcitestyle{notesep={; },square,aysep={},yysep={;}}
\usepackage[usenames,dvipsnames]{color}

\begin{document}
%%%%%%%%%%%%%%%%%% ABSTRACT %%%%%%%%%%%%%%%%%%%%%%%%%%%%%%%%%%
   \begin{abstract}
    
We characterize all infinite-dimensional graded virtual modules for Thompson's sporadic simple group whose graded traces are weight $\frac 32$ weakly holomorphic modular forms satisfying certain special properties. We then use these modules to detect the non-triviality of Mordell--Weil, Selmer, and Tate-Shafarevich groups of quadratic twists of certain elliptic curves.
    
    \end{abstract}
%
%%%%%%%%%%%%%% Title %%%%%%%%%%%%%%%%%%%%
\maketitle
%%%%%%%%%%%%%% Section 1: Introduction %%%%%%%%%%%%%%%%%%%%
\section{Introduction}\label{Intro}

In 1978, McKay and Thompson observed \citep{Thompson} that the first few coefficients of the normalized elliptic modular invariant $J(\z)=q^{-1}+196884q+21493760q^2+864299970q^3+O(q^4),$ a central object in the theory of modular forms, can be written as sums involving the first few dimensions of irreducible representations of the monster group $\mathbb{M}$, e.g., 
	\begin{align}
    \begin{split}
        196884&={\color{white}2\cdot1}1+{\color{white}2\cdot1}{196883}\\
        21493760&={\color{white}2\cdot1}1 + {\color{white}2\cdot1}196883 + 21296876\\
        864299970&={\color{white}{1}}2\cdot1+ {\color{white}{1}}2\cdot196883 +  21296876 + 842609326.
    \end{split}
    \end{align}
This coincidence inspired Thompson's conjecture \citep{Thmp} that there is an infinite-dimensional $\M$-module $V=\displaystyle\bigoplus_{n\geq -1}V_n$ whose graded dimension is $J(\z)$ and whose McKay--Thompson series 
    \be 
    T_g(\z):=\sum_{n\geq -1} \tr(g|V_n)q^n
    \ee
are distinguished functions on the upper half-plane. Conway and Norton \citep{ConwayNorton} explicitly described the relevant McKay-Thompson series, and also christened this phenomenon ``monstrous moonshine.'' Their conjecture was proven by Borcherds \citep{BorcherdsMonster} (building on work by Frenkel, Lepowsky and Meurman \citep{FLM}) in 1992.  In the few decades since the first observations of McKay and Thompson, it has become clear that monstrous moonshine is just the first of a series of similar phenomena encompassing several finite groups and their counterparts in the world of modular forms. 

{Generalized moonshine \citep{carnahangeneralized} (see also, \citep{NortonGeneralized,ConwayNorton,Queen}), for example, relates various subquotients of the Monster to other weight zero modular forms.}
Umbral moonshine \citep{CDHUmbralMoonshine,CDHUmbralMoonshineNiemeier14} (see also \citep{UmbralMoonshineProof-DuncanGriffinOno,Gannon} and \citep{CDHWt1Jacobi}), on the other hand, relates the 23 umbral groups (each of which is a quotient of the automorphism group of one of the 23 Niemeier lattices) to weight $\frac12$ \textit{mock modular forms.} Thompson moonshine, conjectured by Harvey and Rayhaun \citep{HarveyRayhaun} in 2015 {and proven by Griffin and Mertens in \citep{MichealsGriffinMertensThompson},} involves Thompson's sporadic simple group \Th, and certain weight $\frac12$ modular forms. {(We remark here that the Thompson group, being a subgroup of the Monster, also appears in the generalized moonshine setting mentioned above. For the purpose of this paper, ``Thompson moonshine'' refers to the Harvey and Rayhaun version.)}

Recently, in \citep{DuncanMertensOno-NaturePariahMoonshineOnan,DuncanMertensOno-OnanArxiv}, Duncan, Mertens, and Ono discovered the first instance of moonshine for the O'Nan group, one of the so-called pariah groups (i.e., a sporadic simple group which is not a subquotient of the monster group), where the functions involved are modular forms of weight $\frac32.$ Their work is not only a contribution to the theory of moonshine, it also serves another important purpose: In the same paper, they use their O'Nan-module to study properties of quadratic twists of certain elliptic curves and thus use moonshine to provide insight into objects that are central to current research in number theory. 

While number theory's contribution to moonshine is ubiquitous and irrefutable, O'Nan moonshine is one of the first instances where we see moonshine's direct contribution to number theory. Such a role-reversal is our primary motivation for this work. 

We begin this work by proving the existence of a family of infinite-dimensional graded $Th$-modules whose McKay--Thompson series are weight $\frac32$ modular forms that satisfy certain properties (see \Cref{thm:main}). The techniques we use to prove this are similar to ones used in Griffin and Mertens' work \citep{MichealsGriffinMertensThompson} to prove the Thompson moonshine conjecture \citep{HarveyRayhaun}. (These techniques were first suggested by Thompson, and subsequently used by Atkin, Fong and Smith \citep{AtkinFongSmith,Smith} to prove monstrous moonshine abstractly.) On the other hand, our McKay--Thompson series are weight $\frac32$ modular forms (in contrast to the weight $\frac12$ forms of \citep{HarveyRayhaun}) and the role played by theta functions in their paper is taken up by weight $\frac32$ cusp forms in ours. The involvement of weight $\frac32$ cusp forms allows us to employ an approach similar to Duncan, Mertens, and Ono (in \citep{DuncanMertensOno-NaturePariahMoonshineOnan,DuncanMertensOno-OnanArxiv}): We exploit the existing relationship between these forms and elliptic curves to study geometric invariants of various elliptic curves. This is the content of \Cref{Ell14,Ell19}.

Our result regarding the existence of a family of Thompson modules is, in fact, a classification result. We classify all infinite-dimensional graded modules $W=\bigoplus W_n$ (see \Cref{thm:main}) for the Thompson group whose McKay--Thompson series take the form 
    \be 
        \Fg(\z):=6q^{-5}+\sum_{n>0}\tr(g|W_n)q^n 
    \ee 
and satisfy the following properties (cf. \Cref{MainProp}): 
    \begin{enumerate}
        \item For each $ g \in \Th,$ the corresponding McKay--Thompson series $\Fg(\z)$ is a weight $\frac32$ weakly holomorphic modular form of a specific level and multiplier system, and satisfies the Kohnen plus space condition. 
        \item Each McKay--Thompson series $\Fg(\z)$ has integer coefficients and is uniquely determined --- up to the addition of certain cusp forms --- by its polar parts at the cusps, which are specified in a uniform way. (See \Cref{secfunc} for details.)
    \end{enumerate}
We note here that properties (1) and (2) listed above ensure that the functions $\Fg(\z)$ are, up to the addition of cusp forms, simply {\Rsums} projected to the Kohnen plus space (see \Cref{secPrelim} for background on {\Rsums}). 

The connection between {\Rsums} and moonshine was first proposed in \citep{DuncanFrenkel}, where the McKay--Thompson series that appear in monstrous moonshine were characterized completely in terms of {\Rsums} of weight 0. In particular, it was shown that the so-called \textit{genus-zero property} of monstrous moonshine is equivalent to the fact that the McKay--Thompson series of the Monster module coincide (up to a constant) with corresponding Rademacher sums of weight 0. It was later argued in \citep{CDrsums,CDHUmbralMoonshine} that the correct analogue of the genus zero property in the case of Umbral (and Mathieu) moonshine is that the corresponding McKay--Thompson series must coincide with the relevant Rademacher sums in each case (see also \citep{CDHWt1Jacobi,Duncan2019FromMonster}). Here we take this perspective and hence consider it natural, from the point of view of moonshine, to ask for our McKay--Thompson series to satisfy the properties listed above. 

To prove our classification result, we first construct spaces of weakly holomorphic modular forms of the appropriate level and multiplier for each $g \in \Th.$ We use {\Rsums} and eta-quotients to do this. Since we can explicitly compute the Fourier coefficients of these forms at various cusps, we can restrict our attention to the subspace of forms that satisfy properties (1) and (2). For a collection of these forms to be the McKay--Thompson series of a virtual module (as in \Cref{thm:main}), they must satisfy congruences modulo certain powers of primes that divide the order of the Thompson group (see \Cref{sec:I}). A complete description of these congruences can be obtained using Thompson's reformulation (\citep{Smith}) of Brauer's characterization of generalized characters. We prove that our alleged McKay--Thompson series satisfy the congruences mentioned above in \Cref{sec:I}. We note here that it would be interesting to consider the analogous classification for the O'Nan group, building on the work already done in \citep{DuncanMertensOno-OnanArxiv}. 

Once we have proven the existence of the Thompson modules, we use their properties to help detect the non-triviality of Mordell--Weil, Selmer, and Tate--Shafarevich groups of quadratic twists of certain elliptic curves (see \Cref{Ell14,Ell19}). To state our main results, we let $E$ be an elliptic curve over $\Q,$ and for $d<0$ a fundamental discriminant, we let $E^d$ denote the $d^{th}$ quadratic twist of $E.$ (We refer the reader to \citep{Milne,Silverman1} and \citep{Tate} for background on elliptic curves.) We further let $E(\Q)$ denote the set of $\Q$-rational points on $E,$ i.e., the set of points on $E$ whose coordinates are rational numbers. Then, $E(\Q)$ has the structure of a finitely generated abelian group (by the Mordell--Weil theorem \citep{Silverman1}), i.e., $E(\Q)=\Z^r\oplus E(\Q)_{\rm tor}.$ Here $r\in \Z_{\geq }0$ is called the (algebraic) rank of $E,$ and $E(\Q)_{\rm tor}$ is a finite abelian group. Computing the rank of a general elliptic curve is considered a hard problem in number theory. 

One way of approaching this problem is to study the $L$-function associated to $E,$ defined for $s\in\C$ with $\Re(s)> \frac32$ by \be\label{eqn:Lfunction} L_E(s)=\prod_{p \text{ prime}}  \left(1-a_p p^{-s} +\epsilon(p)p^{1-2s}\right)^{-1}.\ee
Here, $a_p:=p+1-\#E(\mathbb{F}_p)$ where $E(\mathbb{F}_p)$ is the group of $\mathbb{F}_p$-rational points on the mod $p$ reduction of $E$, and $\epsilon(p)\in\{0,1\}$ (see \citep{Silverman1} for a precise definition and further discussion on $L$-functions). The (weak) \emph{\BSD} states that the order of vanishing of $L_E(s)$ at $s=1$ is equal to the rank $r$ of $E.$ Thus, studying the behavior of $L_E(s)$ near $s=1$ is one way of tackling the problem of computing the rank of $E.$

Note that \textit{a priori,} the product in \cref{eqn:Lfunction} only converges for $\Re(s)> \frac32,$ so part of the conjecture is that $L_E(s)$ can be analytically continued to a function that converges near $s=1.$ This part follows from the Modularity Theorem (proven in \citep{ModularityTheorem} by extending the results of \citep{WilesModularity1,WilesModularity2}), which states that corresponding to every elliptic curve over $\Q,$ there exists a weight 2 newform of level equal to the conductor of $E,$ whose $L$-function coincides with $L_E(s).$ The $L$-function associated to a cusp form $f=\sum^\infty_{n=1}a_nq^n$ of weight $k$ is defined on $\Re(s)> 1+\frac k2$ and extends analytically to a holomorphic function on $\C$ \citep{AtkinLehnerHecke}. Hence it is possible to talk about the behavior of $L_E(s)$ as $s\to 1,$ for example. The {\BSD} has been proven in special cases (see for example \citep{BhargavaShankar}). In particular, it is known that if $L_E(1)\not=0,$ then $E(\Q)$ is finite \citep{GrossZagier,KolyvaginBSD}. We will use this result to prove our first theorem about elliptic curves.

To state our results, we let \[\F(\z)=6(q^{-5}+85995q^3 - 565760q^4 + 52756480q^7 - 190356480q^8+O(q^{11}))\] be the unique weakly holomorphic modular form of weight $\frac32$ and level $4$ in the plus space whose Fourier expansion is of the form $\F(\z)=6q^{-5}+O(q).$ (Note that $\F(\z)$ is relevant for us because for each of the graded $\Th$-modules $W=\bigoplus_{n>0}W_n$ described in \Cref{thm:main}, the graded dimension of $W$ is $\F(\z).$) Let $c(d)$ denote the $|d|^{th}$ Fourier coefficient in the expansion of $\F(\z).$ We denote by $\left(\frac mn\right)$ the usual Kronecker symbol \citep{kronecker}, then we have the following theorem.
    \begin{theorem}\label{Ell19}
    Let $d<0$ be a fundamental discriminant which satisfies $\left(\frac{d}{19}\right)=-1.$ Let $E$ be an elliptic curve of conductor 19, and let $E^d$ denote the $d^{th}$ quadratic twist of $E$. If $c(d)\not\equiv0 \pmod{ 19},$ then the Mordell--Weil group $E^d(\Q)$ is finite. 
    \end{theorem}
We can state a stronger result for elliptic curves of conductor 14, one which depends on a local version of the strong form of {\BSD.} To do this, we first establish some notation.

For $E$ an elliptic curve of rank $r,$ let $\#E(\Q)_{\rm tor}$ denote the order of the torsion subgroup of the $\Q$-rational points of $E$ and let $\#\Sha(E)$ be the order of the Tate--Shafarevich group of $E.$ Then the strong form of the {\BSD} states the following \citep{BSDWilesClay}. 
    \begin{conjecture}[The {\BSD}]\label{BSD}
    The rank $r$ of an elliptic curve $E$ over $Q$ equals the order of vanishing of $L_E(s)$ at $s=1.$ Moreover, we have %
        \[\label{eqn:bsd} \frac{L^{(r)}_{E}(1)}{r!\Omega({E})}=\# \Sha(E)~ \frac{\Reg(E)\prod_lc_l(E)}{\left(\#E(\Q)_{\rm tor}\right)^2},
        \]
    where $L_{E}^{(r)}(s)$ is the $r^{th}$ derivative of $L_E(s).$
    \end{conjecture}
The left-hand side of \Cref{eqn:bsd} ties together the rank of the elliptic curve and the value of the $r^{th}$ derivative of the $L$-function at $s=1.$ It has been proven to always be a rational number \citep{AgasheBSD}. On the right-hand side, all quantities except $\# \Sha(E)$ are effectively computable invariants of an elliptic curve (see \Cref{sec:Appl} for details). The Tate-Shafarevich group $\Sha(E),$ however, is not even known to be finite except for special cases. 

So we instead turn to a relatively simple object. Given $p$ a prime, we define $\Sha(E)[p]$ to be the set consisting of elements of the Tate--Shafarevich group with order dividing $p$. Then, $\Sha(E)[p]$ is finite for all $p.$ This follows from the finiteness of a related object: 
the $p$-Selmer group $\Sel_p{(E)}$ of $E$, which fits into a short exact sequence
\[1\rightarrow E(\Q)/pE(\Q) \rightarrow \Sel_p(E)\rightarrow \Sha(E)[p]\rightarrow 1.\]
The Selmer group of an elliptic curve over $\Q$ is known to be finite for every $p$ (see \citep{Silverman1}), so we have that $\Sha(E)[p]$ is finite for every $p.$ Also by the short exact sequence, we get that if $\Sel_p(E)$ is trivial for any prime $p$ then $E(\Q)$ is finite, i.e., $E$ has rank 0.

In this work, we will use the family of Thompson modules whose existence is proven in \Cref{thm:main} to develop a criterion to check whether the $p$-Selmer groups of quadratic twists of elliptic curves of conductor 14 are trivial. More precisely, we let $E$ be an elliptic curve over $\Q,$ and for each $d<0$ a fundamental discriminant, let $E^d$ denote the $d^{th}$ quadratic twist of $E.$ As we shall prove in \Cref{thm:main}, there exists an infinite-dimensional graded $\Th$-module $W=\bigoplus_{n>0}W_n$ whose McKay--Thompson series $\Fg(\z)$ satisfy properties (1) and (2) as above. Then we have the following theorem. 
    \begin{theorem}\label{Ell14}
        Let $d<0$ be a fundamental discriminant for which $\left(\frac{d}{7}\right)=-1$ and $\left(\frac{d}{2}\right)=1.$ Let $E$ be an elliptic curve of conductor 14, and let $g$ denote an element of order 14 in $Th.$ If $\tr(g|W_{|d|})\not\equiv 0 \pmod{49},$ then the Mordell--Weil group $E^d(\Q)$ is finite and $\Sha(E^d)[7]$ is trivial. 
        If, on the other hand, $\tr(g|W_{|d|})\equiv 0 \pmod{49}$ and $\tr(g|W_{4})\not\equiv 43 \pmod{56},$ then $\Sel_7(E^d)$ is non-trivial, and if $L_{E^d}(1)$ is non-zero then so is  $\Sha(E^d)[7].$ 
    \end{theorem}
This is akin to Theorem 1.4 of \citep{DuncanMertensOno-OnanArxiv}, and our proof will follow along similar lines.  One notable difference is that there is no dependence on (generalized) class numbers in the corresponding congruences in our case. We can also write down an analogous statement for elliptic curves of conductor 19, but the techniques we use to prove \Cref{Ell14} do not apply in this case (cf. \Cref{Skinner}), so it is conditional upon the (strong form of the) {\BSD.} 

We now describe a sketch of the proof of \Cref{Ell19,Ell14}. Let $p \in \{7, 19\}$ be the relevant prime in either statement and fix $W=\bigoplus_{n>0} W_n$ to be a virtual Thompson module whose McKay--Thompson series $\Fg(\z)$ satisfies the properties listed in \Cref{thm:main}.  We then write each $\Fg(\z)$ for $[g]\in\{14A, 19A\}$ as a sum of traces of singular moduli (cf. \Cref{tracesofsingularmoduli}) and weight $\frac32$ cusp forms. This expression combined with the condition on $d$ in the statement of \Cref{Ell19} gives us that the congruence in the statement holds if and only if the relevant cusp form coefficient is divisible by $p=19.$ Thus, if the congruence in the statement of \Cref{Ell19} does not hold, then the cusp form coefficient is not divisible by 19, and we can employ a corollary (cf. \Cref{kohnen}) of Kohnen's work \citep{Kohnen85} to show that this means $L_{E^d}(1)\not=0.$  Finally, Kolyvagin's work shows that $E^d(\Q)$ is finite. This completes the proof of \Cref{Ell19}.
For \Cref{Ell14}, we first consider the case that $\tr(g_{14}|W_{|d|})\not\equiv0 \pmod{49}.$ The expression for $\Fg(\z)$ in terms of traces of singular moduli and cusp forms implies that the relevant cusp form coefficient is not divisible by $p=7$. We can utilize Kohnen's work again to conclude that $\ord_p(\frac{L_{E^d}(1)}{\Omega({E^d})})>0.$ At this point, we use work of Skinner and Urban (\Cref{Skinner}) which connects $\ord_p(\frac{L_{E^d}(1)}{\Omega({E^d})})$ to the non-triviality of the $p$-Selmer and Tate-Shafarevich groups of $E^d$ to prove the theorem. A similar argument applies if we assume that $\tr(g_{14}|W_{|d|})\equiv0 \pmod{49}$ and $\tr(g_{14}|W_{|4|})\not\equiv43 \pmod{49}.$ 

The rest of this paper is organized as follows. In Section \ref{secPrelim} we set up notation and define the terms that appear in \Cref{MainProp} and  \Cref{thm:main}. In \Cref{secfunc} we uniquely characterize the modular forms $\Fg(\tau)$ that satisfy properties (1) and (2) above and prove \Cref{MainProp}. In \Cref{sec:I} we prove \Cref{thm:main}. Finally, in Section \ref{sec:Appl} we prove \Cref{Ell14,Ell19}. 

\section*{Acknowledgments}
The author would like to thank John Duncan for helping initiate this project and his invaluable guidance and support throughout the process. The author thanks Richard Borcherds, John Duncan, Jeff Harvey, Michael Mertens, Brandon Rayhaun, Robert Schneider, and the anonymous referee for helpful comments on an earlier draft.

%%%%%%%%%%%%%% Section 2: Background and Notation %%%%%%%%%%%%%%%%%%%%

\section{Background and Notation}\label{secPrelim}

Throughout this paper, we use the notation $e(x)=e^{2 \pi i x}$ and $q=e(\tau)$ with $\tau$ in the upper half-plane, which we denote $\HH$. We also use $\left(\frac{m}{n}\right)$ to denote the Kronecker symbol \citep[Algorithm 1.4.10]{kronecker}. We will use the ATLAS \citep{ATLAS} notation for conjugacy classes of $Th$, and understand $nAB$ to mean $nA\cup nB$.

\subsection{Rational Characters.}\label{Prelim:ratchar}

We define the \textit{rational conjugacy class} of an element $g \in Th,$ denoted by $[g],$ to be the set of all elements that are conjugate to an $n^{th}$ power of $g$ where $n$ is relatively prime to the order of $g$. (In particular, this contains the conjugacy class of $g$ as a subset.) 

We recall that a \textit{rational character} of a group $G$ is a character afforded by a $\Q G$-module. By \citep[Lemma 39.4]{Curtis}, if $g$ and $h$ are in the same rational conjugacy class, and $\phi:G\to\Q$ is a rational character, then $\phi(g)=\phi(h).$ We note that if $\phi$ is a rational character of $Th, $ then $\phi(g)$ is an algebraic integer lying in $\Q,$ so in fact $\phi(g)\in \Z.$ 
%We call a rational character $\phi$ of $G$ an\textit{ irreducible rational character} if the corresponding $\Q G$-module cannot be written as a sum of other $\Q G$-modules.

In this paper, we will consider the irreducible rational characters of $Th.$ To describe these, we first define a few things. Let $G$ be a finite group and $V$ be a $\C G$-module. For each field automorphism $\gamma:\C\to\C,$ there exists a unique (up to isomorphism) representation $V^\gamma$ with character \(\chi_{V^\gamma(g)}=\gamma \chi_{V}(g).\) We call $V^\gamma$ a Galois-conjugate of $V.$ Then, we have the following proposition.

\begin{proposition}(see \citep[Theorem 74.5]{Curtis}, for example.)
Let $V_1(=V),V_2\dots, V_n$ be the distinct Galois-conjugates of an irreducible $\C G$-module $V.$ Then there exists a natural number $m_V$ such that $m_V(V_1\oplus\dots\oplus V_n)$ is the complexification of an irreducible $\Q G$-module. Furthermore, each irreducible $\Q G$-module $W$ arises in this way from a unique Galois-class of irreducible $\C G$-modules, i.e. \[W\otimes_\Q \C\simeq m_V(V_1\oplus\dots\oplus V_n).\]
\end{proposition}
The number $m_V$ is called the (rational) Schur index of $V.$ By \citep[Section 7]{Feit}, the Schur index is 1 for each irreducible representation of $Th.$ Thus, we can read off the 39 irreducible rational characters of $Th$ directly from the character table. We denote these by $\chi_1,\chi_2,\dots,\chi_{39}.$

\subsection{Mock Modular Forms} 
To prove the existence of the Thompson modules in \Cref{thm:main}, we first have to construct weakly holomorphic modular forms of weight $\frac 32$ with the appropriate level and multiplier. Recall that a \textit{weakly holomorphic modular form} is a function on the upper half-plane that transforms like a modular form, is holomorphic on the upper half-plane and meromorphic at the cusps. One way of constructing spaces of weakly holomorphic forms is to use {\Rsums,} which are \emph{a priori} mock modular forms, and then restrict to the subspace of forms with vanishing shadow. Here we recall the definitions and basic facts that we will need from the theory of mock modular forms and Rademacher sums to describe these functions. We refer the reader to \citep{Kensbook,BlackHoles,KenRamanujan} for more on mock modular forms.
     
Let $k\in \frac{1}{2} \Z$ and $\Gamma$ be a subgroup of $\SL_{2}(\R)$ containing $\pm I$ such that $\Gamma$ is commensurable with $\SLZ.$ For $\gamma=\abcd,$ write $\gamma\z$ for $\frac{a\z+b}{c\z+d},$ and define $\mathrm{j}(\gamma,\tau)=(c\tau+d)^{-2}.$ We call a function $\psi:\Gamma\to \C$ a \emph{multiplier system} for $\Gamma$ of weight $k$ if 
    \be
        \psi(\gamma_1\gamma_2)\mathrm{j}(\gamma_1\gamma_2,\tau)^{\frac k2}
        =\psi(\gamma_1)\mathrm{j}(\gamma_1,\gamma_2\tau)^{\frac k2}\psi(\gamma_2)\mathrm{j}(\gamma_2,\tau)^{\frac k2} 
    \ee
for each $\gamma_1,\gamma_2\in \Gamma$, where we choose the principal branch of the logarithm to define the exponential $x \mapsto x^s$ in case $s$ is not an integer.

In this paper we will consider multiplier systems of the form 
    \begin{equation}\label{psi}
        \psi_{4N,v,h}\left(\gamma\right):=e\left(-\frac{vcd}{4Nh}\right),
    \end{equation}
where $\gamma=\abcd\in\Gamma_0(4N)$ and $v,h$ are integers with $h|\gcd(N,24)$. 
    
Recall that $\Gamma_0(N)$ is the congruence subgroup 
    \be
        \Gamma_0(N):=\left\{\left.\gamma=\bigabcd\in\SLZ ~\right|~ c\equiv0 \pmod{N}\right\}
    \ee
of the full modular group $\SLZ$. We can now define the $(k,\psi)$-action of $\gamma \in \Gamma_0(N)$ on a smooth function $f:\HH\rightarrow\C$ by   
    \be
        \left(f|_{k,\psi}\gamma\right) (\tau):=\begin{cases} \psi(\gamma)\mathrm{j}(\gamma,\z)^{\frac k2}f\left(\gamma\tau\right) & \text{if }k\in\Z \\
        \left(\frac{c}{d}\right)\eps_d^{2k}\psi(\gamma){{\mathrm{j}(\gamma,\tau)}}^{\frac  k2}f\left(\gamma\tau\right) & \text{if }k\in\frac 12+\Z, \end{cases}
    \ee  
where
    \be
        \eps_d:=\begin{cases} 1 & d\equiv 1\pmod{4},\\ i & d\equiv 3\pmod 4.\end{cases}
    \ee
and we assume $4|N$ if $k\notin\Z$.    
    \begin{definition}
    A \emph{harmonic (weak) Maa{\ss} form} of \emph{weight} $k\in\tfrac 12\Z,$ \emph{level} $N$ and \emph{multiplier system} $\psi$, is a smooth function $f:\HH\rightarrow\C$ on the upper half-plane that satisfies the following properties: \begin{enumerate}
    \item It is invariant under the $(k,\psi)$-action by all $\gamma \in \Gamma_0(N)$ and $\tau=u+iv\in\HH.$ 
    \item It is annihilated by the \emph{weight $k$ hyperbolic Laplacian},
    \be\Delta_k f:=\left[-v^2\left(\frac{\partial^2}{\partial u^2}+\frac{\partial^2}{\partial v^2}\right)+ikv\left(\frac{\partial}{\partial u}+i\frac{\partial}{\partial v}\right)\right] f\equiv 0.\ee
    \item There is a polynomial $P_f(q^{-1})$ such that $f(\tau)-P_f(e^{-2\pi i\tau})=O(e^{-cv})$ for some $c>0$ as $v\to\infty$. Analogous growth conditions are required at all cusps of $\Gamma_0(N)$. 
    \end{enumerate}
    \end{definition}
We denote the space of harmonic Maa{\ss} forms of weight $k$, level $N,$ and multiplier $\psi$ by $H_k({\Gamma_0(N)},\psi)$, and we omit the multiplier if it is trivial.

Bruinier and Funke first introduced harmonic Maa{\ss} forms in \citep{BF04}. We are going to need the following two results from their paper.
    \begin{lemma}\label{lem:split}\citep[equations (3.2a) and (3.2b)]{BF04}
        Let $f\in H_{k}{(\Gamma_0(N),\psi)}$ be a harmonic Maa{\ss} form of weight $k\neq 1$ such that $\psi(\left(\begin{smallmatrix} 1 & 1 \\ 0 & 1\end{smallmatrix}\right))=1$. Then there is a canonical splitting
            \begin{equation}\label{eq:split}
                f(\tau)=f^+(\tau)+f^-(\tau),
            \end{equation}
        where for some $m_0\in\Z$ we have the holomorphic part, 
            \be f^+(\tau):=\sum\limits_{n=m_0}^\infty c_f^+(n)q^n,\ee
        and the non-holomorphic part,
            \be f^-(\tau):=\sum\limits_{\substack{n=1}}^\infty \overline{c_f^-(n)}n^{k-1}\Gamma(1-k;4\pi nv)q^{-n}.\ee
        Here $\Gamma(\alpha;x)$ denotes the upper incomplete Gamma function.
    \end{lemma}

We call the holomorphic part of a harmonic Maa{\ss} form a \textit{mock modular form.} {Let $M_{k}^!{(\Gamma_0(N),\psi)}$ denote the space of weakly holomorphic modular forms of weight $k,$ level $N,$ and multiplier system $\psi.$} Then we have the following proposition.
    
\begin{proposition}\label{shadow} (See \citep[Proposition 3.2]{BF04})
    The operator
    \be\xi_{k}:H_{k}{(\Gamma_0(N),\psi)}\rightarrow {M_{2-k}\left((\Gamma_0(N),\overline{\psi}\right)},\:f\mapsto\xi_{k}f:=2iv^{k}\overline{\frac{\partial f}{\partial\overline{\tau}}}\ee
    is well-defined and surjective with kernel $M_{k}^!{(\Gamma_0(N),\psi)}$. Moreover, we have that
    \be(\xi_{k}f)(\tau)=-(4\pi)^{1-k}\sum\limits_{n=1}^\infty c_f^-(n)q^n\ee
    and we call \(\xi_{k}f\) the \emph{shadow} of (the holomorphic part of) $f$.
\end{proposition}

{Thus, in particular, a} mock modular form is a weakly holomorphic modular form if it has a vanishing shadow. We will construct the desired space of weakly holomorphic forms by first constructing mock modular forms of the appropriate level and weight and then showing that they have vanishing shadows. 
    
\subsection{\Rsums} To construct the relevant mock modular forms for the proof of \Cref{thm:main}, we need to recall some facts about Rademacher sums and Rademacher series. See \citep{CDRsumsMathieuHolographic, CDrsums, DuncanFrenkel} for more details. 
    
    Let $\Gamma_\infty:=\left\{\pm\left(\begin{smallmatrix} 1 & n \\ 0 & 1 \end{smallmatrix}\right)\: : \: n \in\Z\right\}$ denote the stabilizer of $\infty$ in $\Gamma_0(N).$ Then one can define the Rademacher sum of weight $k\geq1$, level $N$, multiplier system $\psi$ and index $\mu,$ by 
    \be R^{[\mu]}_{k, N, \psi}(\tau):=\lim\limits_{K\to \infty} \sum_{\gamma\in\Gamma_\infty\setminus \Gamma_{K,K^2}(N)} q^\mu|_{k,\psi}\gamma\ee 
    where \[\Gamma_{K,K^2}(N):=\left\{\abcd\in\Gamma_0(N)\::\:|c|<K\text{ and }|d|<K^2\right\}\] and $\mu \in \Z+$${\frac{-i}{2\pi}{\log\left(\psi\left(\begin{smallmatrix} 1 & 1 \\ 0 & 1 \end{smallmatrix}\right)\right)}}.$
    
When convergent, Rademacher sums define mock modular forms of level $N,$ weight $k,$ and multiplier system $\psi.$ We will use the following important facts from the theory of Rademacher sums, which we condense in one lemma.  We will use the following important facts from the theory of Rademacher sums, which we condense in one lemma. 
     \begin{lemma}\label{lemma:Rad}(See \citep[Theorem 2.5]{DuncanMertensOno-OnanArxiv}, for example.) Let $\mu\leq 0.$ Assuming locally uniform convergence, the Rademacher sum $R^{[\mu]}_{k, 4N, \psi}(\tau)$ for $k\geq 1$ defines a mock modular form of weight $k\in\Z+\frac 12,$ for $\Gamma_0(4N)$ and multiplier $\psi$ whose shadow {is given by a constant multiple of the Rademacher sum $R^{[-\mu]}_{2-k, 4N, \overline{\psi}}(\tau).$} The completion of $R^{[\mu]}_{k, 4N, \psi}(\tau)$ to a harmonic Maa{\ss} form has a pole of order $\mu$ at the cusp $\infty$ and vanishes at all other cusps.
 \end{lemma}
In this paper, we will be looking at Rademacher sums of weight $\tfrac 32$ for $\Gamma_0(4N)$ with multiplier $\psi_{4N,v,h}$ and index $\mu<0.$ In this particular case, it has been proven in \citep{CDRsumsMathieuHolographic} that the sums converge locally uniformly and define holomorphic functions on $\HH.$  

\subsection{Kohnen Plus Space Condition}\label{subsecproj} One of the properties that we want our candidate McKay--Thompson series to satisfy is to lie in Kohnen's plus space. {Let $S_{k}(\Gamma_0(N))$ denote the space of weight $k$ cusp forms for $\Gamma_0(N).$} 
    Kohnen's plus space was first introduced {by Kohnen (see \citep{Kohnen80,Kohnen82,Kohnen85}) as the subspace of $S_{k+\frac12}(\Gamma_0(4N))$} which consists of all forms whose Fourier coefficients are supported on exponents $n$ with $n\equiv 0,(-1)^k\pmod 4$. We extend this idea to all modular forms and harmonic Maa{\ss} forms as follows: We say that a function $f$ in $H_{k+\frac 12}(\Gamma_0(4N))$ (resp. in $M_{k+\frac 12}^!(\Gamma_0(4N))$) \textit{satisfies the Kohnen plus space condition} if the Fourier coefficients of $f$ are supported on exponents $n$ with $n\equiv 0,(-1)^k\pmod 4$. We denote the space of such forms $H^{+}_{k+\frac 12}(\Gamma_0(4N))$ (resp. $M^{+,!}_{k+\frac 12}(\Gamma_0(4N))).$

    For odd $N$, there is a natural projection operator \(|\pr: S_{k+\frac 12}(\Gamma_0(4N))\rightarrow S_{k+\frac 12}^+(\Gamma_0(4N))\) given in terms of slash operators which extends to spaces of weakly holomorphic modular forms and harmonic Maa{\ss} forms.
    Let $f\in M^{!,+}_{k+\frac 12}(\Gamma_0(4N)),$ where $N$ is an odd integer. Then the projection operator acts on $f$ in the following way \citep{Kohnen85} (see also, \citep{MichealsGriffinMertensThompson}) %Note to self, page 14 of 32 of Kohnen's paper.
    \[(f|\pr)(\tau)=(-1)^{\lfloor{\frac{k+1}2}\rfloor}\frac{1}{3\sqrt2}\sum^{2}_{v=-1}\left(f|\left(\begin{smallmatrix} 4(1+Nv) & 1 \\ 8Nv & 4 \end{smallmatrix}\right) \right)(\tau)+\frac 13 f(\tau).\]

    The action of this projection operator on principal parts of harmonic Maa{\ss} forms is described in the following lemma (see \citep[Lemma 2.9, 2.10]{MichealsGriffinMertensThompson} and \citep[Lemma 2.6]{DuncanMertensOno-OnanArxiv}).
    \begin{lemma}\label{plusspace}
    Let $N$ be odd and $f\in H_{k+\frac 12}(\Gamma_0(4N))$ for some $k\in\N$, such that 
    \be f^+(\tau)=q^{-m}+\sum_{n=0}^\infty a_nq^n\ee
    for some $m>0$ with $-m\equiv 0,(-1)^k\pmod 4,$ and suppose that $f$ has a non-vanishing principal part only at the cusp $\infty$ and is bounded at the other cusps of $\Gamma_0(4N)$. Then the projection $f|\pr$ of $f$ to the plus space has a pole of order $m$ at $\infty,$ a pole of order $\frac m4$ either at the cusp $\frac 1N$ if $m\equiv 0\pmod 4$ or at the cusp $\frac1{2N}$ if $-m\equiv(-1)^k\pmod 4$ and is bounded at all other cusps.
    \end{lemma}
     For even $N,$ we have the following lemma for the {\Rsums} that we consider in this paper. 
       \begin{lemma}\label{evenrad}
    For even $N,$ the {\Rsum} $R^{[-5]}_{\frac 32,4N,\psi}(\tau)$ satisfies the Kohnen plus space condition. 
    \end{lemma}
    \begin{proof}
    This is an immediate consequence of Lemma 2.10 of \citep{MichealsGriffinMertensThompson}. 
    \end{proof} 

 For even $N,$ we define the projection operator {|\pr} to be the following sieving operator:
 Let $f(\tau)=\sum^\infty_{n=n_0}c(n)q^n$ be modular of weight $k+\tfrac12$ where $k\in \N,$ and level $4N,$ where N is even, then we define  \[f|\pr=\sum^\infty_{\substack{n=n_0\\n\equiv0,(-1)^k\pmod 4}}c(n)q^n.\]
 By \Cref{evenrad}, we have \(R^{[-5]}_{\frac 32,4N,\psi}(\tau)|\pr=R^{[-5]}_{\frac 32,4N,\psi}(\tau)\) if $N$ is even.
\subsection{Eta-Quotients}\label{sec:eta}
    For the rational conjugacy classes $[g] \in \{21A, 30AB\},$ it is convenient to use eta-quotients instead of Rademacher sums.
    
    Recall that an eta-quotient is defined to be a function of the form
    \be
      f(\tau) = \prod_{\delta | N} \eta(\delta \tau)^{r_{\delta}},
    \ee
    where $r_{\delta} \in \Z$ and $  \eta(\tau) := q^{\frac{1}{24}}  \prod_{n=1}^{\infty} (1-q^{n})$ is the Dedekind eta function.
     As a consequence of the product definition for
    $\eta(\tau)$, any eta-quotient is non-vanishing on $\HH$. We will need the following lemma from \citep{magiceta} to construct eta-quotients which vanish only at a specific cusp.
 \begin{lemma} \label{magicetas}(\citep[Lemma 14]{magiceta})
    Let $ N\in \N$, then for each divisor $d$ of $N$, there exists $k_d \in \N $ and a corresponding eta-quotient $E_{d,N}(\tau) \in M_{k_{d}}(\Gamma_{0}(N))$ such that $E_{d,N}$ vanishes only at the cusp $\tfrac cd$.

    \end{lemma}
   
    The proof of \Cref{magicetas} is constructive, and \texttt{MAGMA} \citep{Magma} code implementing it can be found at \url{http://users.wfu.edu/rouseja/eta/}. We write $E_N(\tau)$ for the holomorphic eta-quotient $E_{N,N}(\tau)$ that is produced by this code. In \Cref{Fwheta} we will use the explicit construction of eta-quotients $E_{N}(\tau)$ to construct weakly holomorphic modular forms $\Fwh(\tau)$ for $[g] \in \{21A, 30AB\} $. 

For completeness, we recall here the modular transformation law of the Dedekind eta-function, 
    \be\label{etamodulartransformation}
        \eta\left(\frac{az+b}{cz+d}\right) = e\left(\frac{a+d}{24c}+\frac{s(-d,c)}{2}+\frac 38 \right) (cz+d)^{\frac{1}{2}} \eta(z).
    \ee%
Here, $\gamma=\left( \begin{matrix} a & b \\ c & d \end{matrix} \right)\in \SLZ$ and $s(d,c)$ is the Dedekind sum,
    \be
        {s(d,c)=\sum_{r=1}^{c-1}\frac{r}{c}\left(\frac{dr}{c}-\left\lfloor\frac{dr}{c}\right\rfloor-\frac{1}{2}\right).}
    \ee
Using \cref{etamodulartransformation}, we can explicitly compute Fourier coefficients of $\left(f|_{k,\psi}\gamma\right)(\tau)$ whenever $f$ is an eta-quotient, $\gamma \in \Gamma_0(N)$ and $\psi$ is a multiplier system of weight $k$ and level $N.$ 

%%%%%%%%% SECTION 3: MCKAY THOMPSON SERIES %%%%%%%%%%%%%%%%%
\section{McKay--Thompson Series}\label{secfunc}
To state our main theorem of this section, we associate to each rational conjugacy class $[g]$ of the Thompson group $Th$, the following data:
\begin{enumerate}
    
        \item  Integers $v_g$ and $h_g$ as specified in \Cref{mult}. We use these to define the character $\psi_g:=\psi_{4|g|,v_g,h_g}$ (see \Cref{psi}), where $|g|$ denotes the order of $g$ in $Th.$
        
        \item The space of cusp forms $S_{g}:=S^{+}_{\frac 32}\left(\Gamma_{0}(4|g|), \psi_{g}\right)$ of weight $\frac 32$ in the plus-space which transform under $\Gamma_0\left(4|g|\right)$ with character $\psi_{g}$. We define $d_g$ to be the dimension of this space and let $f_g$ be the $d_g$-tuple $f_g:=(f^{(1)}_g, \dots, f^{(d_g)}_g)$ where $f^{(i)}_g$ is the $i^{th}$ element of the {echelonised} basis of $S_{g}.$ A list of $f^{(i)}_g$'s can be found in \Cref{fcoeff}. 
     	
     	\item For each $f^{(i)}_g$ defined as above, we let $n^{(i)}_g$ and $m^{(i)}_g$ be the integers listed in \Cref{ngmg}. We define $n_g			$ to be the $d_g$-tuple $n_g:=(n^{(1)}_g, \dots, n^{(d_g)}_g)$, and define $M_g$ to be the $d_g \times d_g$ diagonal matrix with entries given by $(m^{(1)}_g, \dots, m^{(d_g)}_g)$.
     	
     	\item Finally, to each rational conjugacy class $[g]$ of $\Th$, we associate integers $a_g(n)$ for each $n>0,$ where $n\equiv0,3 \pmod{4}.$ For $g\not\in\{21A,30AB\},$ the integers $a_g(n)$ are given by \Cref{A}. {For $m\in\{21,30 \},$ we will prove in \Cref{FwhAB} that there exists a weakly holomorphic % $\{m\in 21, 30\}$ we define $a_{mA}(n)$ to be the Fourier coefficients of the unique 
     	modular form in $M^{+,!}_{\frac32}(\Gamma_0(4m), \psi_{mA})$ which has a pole of order 5 at $\infty,$ a pole of order $\frac 54$ at the cusp $\frac{1}{42}$ if $m=21,$ and vanishes at all other cusps. For $g=21A$ we let  $f^{wh}_{mA}(\tau)$ denote the unique such form whose Fourier expansion begins $6q^{-5}-2 q^{4}+4 q^{7}-8 q^{8}+O(q^{11}),$ and we define $a_{21A}(n)$ by setting \[\FwhA=6q^{-5}+\sum_{\substack{n>0 \\ n\equiv 0,3\smod{4}}}^\infty a_{21A}(n)q^n.\]
     	Similarly, we define $a_{30AB}(n)$ to be the Fourier coefficients of $\FwhB\in M^{+,!}_{\frac32}(\Gamma_0(120), \psi_{30}),$ which is the unique form satisfying all properties in \Cref{FwhAB} with the Fourier expansion $6q^{-5}+3q^3 + 3q^8-3q^{11} + O(q^{12}).$} %We will prove the existence of $\FwhA$ and $\FwhB$ in \Cref{FwhAB}.
     	
    \end{enumerate}
For each rational conjugacy class $[g]$ we define $\Lambda_g$ to be the set of all $d_g$-tuples $\lambda_g:=(\lambda^{(1)}_g, \dots, \lambda^{(d_g)}_g) \in \Z^{d_g}$ and obtain the following proposition.
    \begin{proposition}\label{MainProp}
    Fix a rational conjugacy class $[g]$ in $Th.$ Then for each $\lambda_g \in \Lambda_g$, the function    
        \be
        \Fgt:=6q^{-5}+\sum_{0<n} a_{g}(n) q^n+(\lambda_g M_g+n_g)\cdot f_g(\tau)
        \ee%
    is a weakly holomorphic modular form that satisfies the following properties.                       
        \begin{enumerate}[$(a)$]
            \item It lies in $ M^{+,!}_{\frac 32}(\Gamma_0(4|g|), \psi_{g}),$ i.e., $\Fgt$ has weight $\frac32,$ level $4|g|$ with character $\psi_g,$ and satisfies the Kohnen plus space condition.
            \item It has a pole of order 5 at the cusp $\infty$, a pole of order $\frac 54$ at the cusp $\frac 1{2|g|}$ if $|g|$ is odd, and vanishes at all other cusps. 
            \item The Fourier coefficients of $\Fgt$ are integers. 
        \end{enumerate}
    \end{proposition}   
Note that our multiplier system $\psi_g$ is conjugate to the one used by \citep{MichealsGriffinMertensThompson} and \citep{HarveyRayhaun}. This is necessary for \Cref{thm:main} to be true, and is not unexpected since our functions are weight $\frac{3}{2}$ as opposed to the weight $\frac 12$ forms in \citep{MichealsGriffinMertensThompson} and \citep{HarveyRayhaun}.

We will prove \Cref{MainProp} by constructing specific weakly holomorphic forms $\Fwh(\tau)=6q^{-5}+\sum_{0<n} a_{g}(n)\in M^{+,!}_{\frac 32}(\Gamma_0(4|g|), \psi_{g})$ using the theory of {\Rsums } and eta-quotients. We will see from the explicit construction that each $\Fwh(\tau)$ already satisfies properties $(a)-(c)$ listed in \Cref{MainProp} even without the addition of any cusp forms. We need to add the cusp forms $f_g(\tau)$ not for \Cref{MainProp}, but instead for the following theorem.
    \begin{theorem}\label{thm:main}% 
    Assume the above notation and let $\Lambda$ be the set of functions
    $\{ \lambda: g\mapsto \lambda_g\in \Lambda_g\}$. Then, for each $\lambda \in \Lambda$, there exists an infinite-dimensional graded virtual $Th$-module \be W^{\lambda}:=\bigoplus_{\substack{n>0 \\ n\equiv 0, 3 \pmod{4}}} W^{\lambda}_{n}  \ee such that for each rational conjugacy class $[g]$ of $Th,$ the corresponding McKay--Thompson series, % 
        \be
        6q^{-5}+\sum_{\substack{n>0\\ n\equiv0,3 \pmod{4}}}\tr\left(g|W^{\lambda}_{n}\right)q^n  
        \ee %
    is the specific weakly holomorphic modular form $\mathcal{F}^{\lambda}_{g}(\tau)\in  M^{+,!}_{\frac 32}(\Gamma_0(4|g|), \psi_{g})$ described in \Cref{MainProp}. Furthermore, for every infinite-dimensional graded virtual $Th$-module $ W=\bigoplus_{n>0} W_{n}$ for which  the McKay--Thompson series,  \be\mathcal{F}_{g}(\tau)=6q^{-5}+\sum_{\substack{n>0\\ n\equiv0,3 \pmod{4}}}\tr\left(g|W_{n}\right)q^n  \ee satisfies the properties listed in \Cref{MainProp}, there exists a  $\lambda \in \Lambda$ for which $W=W^\lambda$ described as above. 
    \end{theorem}

We will now construct the relevant spaces of modular forms required for the proof of \Cref{MainProp}.
    
\subsection{Using Rademacher Sums}\label{Fwh} For each rational conjugacy class $[g] \notin \{21A, 30AB\}$, consider the function %
    \begin{equation}\label{fgwh}
        \Fwh(\tau)=
        6R^{[-5], +}_{\frac 32,4|g|,\psi_g}(\tau):=
        6\left(R^{[-5]}_{\frac 32,4|g|,\psi_g}|\pr\right)(\tau) 
    \end{equation}
where $|\pr$  is the projection onto the Kohnen plus-space (see \Cref{subsecproj}). Then by \Cref{lemma:Rad}, \Cref{plusspace} and \Cref{evenrad}, each $\Fwh(\tau)$ is a mock modular form in the plus-space of weight $\frac 32,$ level $|g|,$ and multiplier $\psi_g,$ has a pole of order 5 at the cusp at infinity, a pole of order $\frac 54$ at the cusp $\frac 1{2|g|}$ if $|g|$ is odd (forced by the projection to the plus-space, see \Cref{plusspace}), and vanishes at all other cusps. The only thing left to prove here is that each $\Fwh(\tau)$ is, in fact, weakly holomorphic (i.e., has vanishing shadow, see \Cref{shadow}). 
    \begin{lemma}\label{Rademachercoeff} 
        For each rational conjugacy class $[g] \notin \{21A, 30AB\}$ of the Thompson group, the function $\Fwh(\tau)$ (defined in \cref{fgwh}) is in fact a weakly holomorphic modular form, and has Fourier expansion given by
            \be
                \Fwh(\tau)=6q^{-5}+\sum_{\substack{n>0 \\ n\equiv 0,3\smod{4}}}^\infty a_g(n)q^n,
            \ee
        where for $N=|g|,$ we have 
        \begin{align}\label{A}
        a_g(n)&:=\frac{-3\pi}{ N}\left(\frac{-n}{5}\right)^{\frac 14} \sum_{c=1}^\infty \frac{1+\delta_{\rm odd}(Nc)}{c} K_{\frac 32,\psi}(-5,n,4Nc) I_\frac 12\left(\frac{\pi\sqrt{5n}}{Nc}\right).
        \end{align}
        Here, $I_\frac 12$ is the modified Bessel function of the first kind of order $\tfrac 12,$
        \be\delta_{\rm odd}(k):=\begin{cases} 1 & k\text{ odd,} \\ 0 & k\text{ even,}\end{cases}\ee
        and $K_{\frac 32, \psi}$ is the twisted Kloosterman sum 
        \begin{equation}\label{Kloost}
        K_{\frac 32, \psi}(m,n,c):=
        \sum_{d\smod{c}}\psi\left(\left(\begin{smallmatrix} * & * \\ c & d \end{smallmatrix}\right)\right)\left(\frac cd\right)\eps^3_d e\left(\frac{m\overline{d}+nd}{c}\right).
        \end{equation}
        The sum here runs over primitive residue classes modulo $c$, and $\bar{d}$ denotes the multiplicative inverse of $d$ modulo $c.$
        
    \end{lemma}
\begin{proof}
    We have already established that each $\Fwh(\tau)$ is a mock modular form of weight $\frac 32$ for the group $\Gamma_0(4|g|,\psi_g).$ {It suffices to show that $\Fwh(\tau)$ has vanishing shadow. By \Cref{shadow}, the space of possible shadows is $M_{\frac12}(\Gamma_0(4|g|),\overline{\psi_g})\subset M_{\frac12}(\Gamma_0(4|g|h_g)).$ By the Serre-Stark basis theorem \citep{SerreStark} (see also, \citep[Theorem 1.45]{KenOnoWebofModularity}), the (a priori) larger space is generated by theta functions of the form $\theta_{\chi}(k\tau)=\sum_{n\in \Z} \chi(n)q^{kn^2},$ where $\chi$ is an even Dirichlet character and $\chi$ and $k$ depend on $N=|g|.$ We compute the Bruinier-Funke pairing \citep{BF04} to deduce that \[\{\widehat{\Fwh(\tau)},\theta_{\chi}(k\tau)\}\not=0\]
    if and only if $k=5.$ Thus, the only way to get non-vanishing shadow would be if the shadow is a constant multiple of $\theta_{\chi}(5\tau)$ but by \Cref{lemma:Rad}, the shadow is in fact a constant multiple of the Rademacher sum $R^{[+5]}_{\frac 12,4N,\overline{\psi_g}}(\tau).$ Since they transform under $\Gamma_0(4|g|)$ with different characters, this is not possible. }
    
    Thus, each $\Fwh(\tau)$ {has vanishing shadow and} is in fact weakly holomorphic. Computing the coefficients of Rademacher sums in terms of Kloosterman sums and Bessel functions is a standard computation, see for example \citep{CDrsums} (or Proposition 2.7 in \citep{DuncanMertensOno-OnanArxiv}.)
\end{proof}
\subsection{Using Eta-Quotients.}\label{Fwheta} 
For $[g] \in \{21A, 30AB\},$ we will use the eta-quotients $E_{|g|}(\tau)$ in \Cref{magicetas} to compute spaces of weakly holomorphic forms of the desired weight and level. The main result of this section is the following lemma.  
    \begin{lemma}\label{FwhAB}
        For $m \in \{21,30\},$ there exists a weakly holomorphic modular form with integer Fourier coefficients in $ M^{!,+}_{\frac 32}(\Gamma_0(4m, \psi_{mA}))$ which has a pole of order 5 at $\infty$, a pole of order $\frac 54$ at the cusp $\frac 1 {42}$ if $m=21,$ and vanishes at all other cusps. 
    \end{lemma}
Note that if \Cref{FwhAB} is true, such a form satisfies all properties of \Cref{MainProp}, and we can thus define $\FwhA(\tau)$ (resp. $\FwhB(\tau)$) to be the unique such form with Fourier expansion $6q^{-5}-2 q^{4}+4 q^{7}-8 q^{8}+O(q^{11})$ (resp.  $6q^{-5} + 3q^3 + 3q^8-3q^{11} + O(q^{12})$)

    \begin{remark}
        We could have written an analogous statement for each rational conjugacy class $[g]$ in the Thompson group and forgone the discussion about {\Rsums} completely. This would not affect the proof of \Cref{thm:main} at all. However, we need an expression for $\Fwh(\tau)$ for certain classes $[g]\not\in\{21A,30B\}$ in terms of {\Rsums} for the application to elliptic curves. In particular, such expressions for $g\in\{14A,19A\}$ play key roles in the proofs of \Cref{Ell19,Ell14}. {On the other hand, for $g\in \{21A, 30AB\},$ the analogue of the sum in \Cref{A} converges very slowly, so it is more convenient to use eta-quotients instead of Rademacher sums in these cases.}%, because the $q^4$ Fourier coefficient of $R^{[-5]}_{\frac 32,84, \psi_{21A}}(\tau),$ for example, is $\approx-0.37368835251224884511050769.$ For Proposition 3.1 to be true, the Fourier coefficient of $\mathcal{F}^{\lambda}_{21A}(\tau)$ has to be an integer. While it is possible to let $f_g^{wh}=R^{[-5]}_{\frac 32,84, \psi_{21A}}(\tau)$ and add the correct multiple of the cusp form $f^{(1)}_{21A}(q)=q^4+O(q^{11}),$ this would make our $n_g$ and $m_g$ (cf. \Cref{MainProp} and \Cref{ngmg}) non-integral. For ease of exposition, we have taken the eta-quotient route.}
    \end{remark}

\begin{proof}[Proof of \Cref{FwhAB}]
    By \Cref{magicetas}, for each $N>0$ we can construct an eta-quotient $E_{N}(\tau)$ that vanishes only at the cusp $\infty.$ We use $N=21$ and $N=30$ to get eta-quotients: 
    \begin{equation}
        E_{21}(\tau)= \frac{\eta(\tau)\hspace{1pt}\eta(21\tau)^{21}}{\eta(3\tau)^3\hspace{1pt}\eta(7\tau)^7} ~\text{   and   }~ E_{30}(\tau)= \frac{\eta(2\tau)^2\hspace{1pt} \eta(3\tau)^3\hspace{1pt}\eta(5\tau)^5\hspace{1pt}\eta(30\tau)^{30}}{\eta(\tau)\hspace{1pt}\eta(6\tau)^6\hspace{1pt} \eta(10\tau)^{10}\hspace{1pt}\eta(15\tau)^{15}}
    \end{equation}
    of weight 6 and 4, respectively. Now, consider the cusp form space $S_{\frac{15}{2}}(\Gamma_0(84), \psi_{21A}), $ and suppose for now that we can compute a basis for this space explicitly. If so, we can divide each element of the basis by $E_{21}(4\tau)$ to get a generating set $B_{21}$ of forms in $M^!_{\frac 32}(\Gamma_0(84), \psi_{21A})$ whose only (possible) pole is at the cusp $\infty.$ Then, we apply the projection operator to each element of $B_{21}$ to get a generating set of forms in $M^{!,+}_{\frac 32}(\Gamma_0(84), \psi_{21A})$ that are holomorphic away from the cusps at $\infty $ and $\frac{1}{42}$. This generating set turns out to be non-empty.
    
    We can now construct $\FwhA(\tau)$ as a suitable linear combination of elements of this space determined completely by its Fourier expansion $6q^{-5}-2 q^{4}+4 q^{7}-8 q^{8}+O(q^{11}).$
    The same argument works for $\FwhB(\tau)$ if we start with $S_{\frac{11}{2}}(\Gamma_0(120), \psi_{30AB})$ instead. 
    Here again, the set of forms in $M^{!,+}_{\frac 32}(\Gamma_0(120), \psi_{30AB})$ which are holomorphic away from the cusp at $\infty$ turns out to be non-empty. 

     We now describe how to compute the bases for $S_{\frac{15}{2}}(\Gamma_0(84), \psi_{21A})$ and $S_{\frac{11}{2}}(\Gamma_0(120), \psi_{30AB})$ in some detail. We will essentially follow the method described in Proposition 3.1 of \citep{MichealsGriffinMertensThompson}. Let $(m,k) \in \{(21, \frac{15}{2}),  (30, \frac{11}{2})\}.$ Let $f \in S_{k}(\Gamma_0(4m)), \psi_{mA})$ and let \be\vartheta(\tau):=\sum_{n\in \Z}q^{n^2}\in M^{+}_{\frac12}(\Gamma_0(4)).\ee Then $f \vartheta$ lies in $M_{k+\frac{1}{2}}(\Gamma_0(4m), \psi_{mA})\subset M_{k+\frac{1}{2}}(\Gamma_0(4m h_g)).$ Using programs (available at \url{http://users.wfu.edu/rouseja/eta/}) written by Rouse and Webb  one can verify that the space $M_{k+\frac{1}{2}}(\Gamma_0(12m))$ is generated by eta quotients ($h_g=3,$ for both values of $m$). Since we can explicitly compute Fourier expansions of $(g|_{k} \gamma)(\tau)$ for any eta-quotient $g(\tau)$ and $\gamma \in \SLZ$ (see \Cref{sec:eta}), we can thus compute a basis for $M_{k+\frac{1}{2}}(\Gamma_0(4m), \psi_{mA})$ and hence for $S_{k}(\Gamma_0(4m)), \psi_{mA}).$ Alternatively, we can also compute a basis for $M_{k+\frac{1}{2}}(\Gamma_0(4m), \psi_{mA})$ using in-built functions in PARI/GP \citep{PARI2}. This concludes the proof of \Cref{FwhAB}.
\end{proof}
     
We now have an explicit description of $\Fwh(\tau) $ for each rational conjugacy class $[g].$ The next steps in the proof are showing that each $\Fwh(\tau)$ (and thus each $\Fgt$) satisfies all properties listed in \Cref{MainProp} (we do this in \Cref{cusp}), and that the only cusp forms we can add for \Cref{thm:main} to be true are appropriate integer multiples of the elements of $S_g$ for each $g \in Th,$ respectively (cf. \Cref{MainProp}). The latter will follow from our work in \Cref{sec:I}. 
\subsection{Cusp forms}\label{cusp}
We begin by noting that the weakly holomorphic forms $\Fwh(\tau)$ described in \Cref{Fwh,Fwheta} satisfy the properties $(a)-(c)$ listed in \Cref{MainProp}. Also, said properties uniquely determine a weakly holomorphic form up to cusp forms \citep[Lemma 2.4]{DuncanMertensOno-OnanArxiv}. Thus to specify the functions $\Fgt$ completely we have to compute the cusp form spaces $S^{+}_{\frac 32}(\Gamma_0(4|g|), \psi_g)$ for each $[g]$ in $Th.$ 
    \begin{lemma}
        For each rational conjugacy class $[g]$ of $Th,$ the corresponding cusp form space $S_g$ is spanned by the cusp forms given in \Cref{fcoeff}.
    \end{lemma}
\begin{proof}
    We use the same method as in the proof of \Cref{FwhAB} to compute the cusp form spaces. Let $f\in S_\frac 32^+(4|g|,{\psi_g})$ be any cusp form. Then $f \vartheta$ lies in $M_2(4|g|,{\psi_g})\subset M_2(4|g| h_g),$ where the larger space is spanned by eta quotients for each $[g].$ This can be verified using \texttt{MAGMA} code written by Rouse and Webb \citep{magiceta}. We can then use the modular properties of the eta-quotients and the projection onto the plus-space to determine a basis for $M_2(4|g|,{\psi_g})$ and hence $S_g= S_\frac 32^+(4|g|,{\psi_g}).$ The space $S_g$ turns out to be trivial for every \be[g] \not \in \{12D, 14A, 18B, 19A, 20A, 21A, 24AB, \allowbreak 24CD, 28A, 30AB, 31AB, 39AB\}. \ee The Fourier coefficients given in \Cref{fcoeff} are enough to determine each $\Fcf(\tau)$ completely for all other rational conjugacy classes $[g].$ 
\end{proof} 
\subsection{Integer Coefficients.}
    The last thing we need to check in order to prove \Cref{MainProp} is that the functions \be\Fgt=\Fwh(\tau)+(\lambda_g M_g+n_g)\cdot f_g(\tau)  \ee constructed in the preceding section have integer coefficients. We will use Sturm's theorem \citep{Sturm} for this. Note that each of these functions lies in $M^{+,!}_{\frac 32}(\Gamma_0(4|g|), \psi_{g})\subset M^{+,!}_{\frac 32}(\Gamma_0(4|g|h_g)),$ thus if $\nu(\tau)=q^5+O(q^8)$ is a cusp form with integer coefficients in $S^{+}_{2k-\frac{3}{2} }(\Gamma_0(4|g|h_g))$, then, $\Fgt\nu(\tau) $ lies in $M_{2k}(\Gamma_0(4|g|h_g)),$ so we can apply Sturm's theorem to it. Thus, $\Fgt$ has integer coefficients if the first $\frac{k}{6}[\SLZ:\Gamma_0(4|g|h_g)]$ coefficients of $\Fgt\nu(\tau)$ are integers. The largest bound we have to check is less than 1200. The author used PARI/GP \citep{PARI2} to do this computation. 
    
    This concludes the proof of \Cref{MainProp}.
%%%%%%%%%%%%%%%%%%% SECTION 4 INTEGER MULTIPLICITIES %%%%%%%%%%%%%%%%%%%%%%%%%%%%%%%%%%
\section{Proof of Theorem \ref{thm:main}: Integer Multiplicities}\label{sec:I}
To prove \Cref{thm:main}, we have to show that the $\Fgt$'s we described in Section \ref{secfunc} are indeed the McKay--Thompson series of a virtual module of the Thompson group.  
 
This is equivalent to proving that there exist integers $m^{\lambda}_1(n),...,m^{\lambda}_{39}(n)$ such that if $\Fgt=6q^{-5}+\sum_{n\geq 3} \alpha_g^\lamda(n)q^n,$ then for each $n\geq 3$ the Fourier coefficient $\alpha_g^\lamda(n)$ can be written in the form,
    \begin{equation}\label{eq:mults}
        \alpha^{\lambda}_g(n)=\sum_{j=1}^{39} m^{\lambda}_j(n)\chi_j(g),
    \end{equation}
where $\chi_1,\dots,\chi_{39}$ are the  irreducible rational characters of $Th$ (See \Cref{Prelim:ratchar} for a definition of rational character). We say that the function $\omega^{\lamda}_{n}:Th\rightarrow \C,$ defined by $g  \mapsto  \alphg(n)$, is a \textit{virtual rational character} of $Th$ if the above condition is satisfied. Thus, the goal of this section is to prove that $\omega^{\lamda}_{n}$ is a virtual rational character of $Th$ for every $n\geq 3$ and choice of $\lambda \in \Lambda.$ 
    
As explained in \citep{MichealsGriffinMertensThompson}, this is computationally infeasible to prove directly using only Sturm bounds \citep{Sturm}. However, it can be reduced to a finite computation using a variant of Thompson's reformulation (see, for example, \citep{Smith}) of Brauer's characterization of generalized characters. (For another example of a similar computation, see \citep{Gannon}.) To state the result, we first have to define a few things. 
    
For the rest of this section, let $G$ be a finite group and $p$ a fixed prime dividing the order of $G$. Let $\mathscr{C}_G$ denote the set of all rational conjugacy classes of $G.$ We call $[g] \in \mathscr{C}_G$  $p$-\emph{regular} if the order of $g$ is coprime to $p$. Let $K_p$ denote the set of all $p$-regular classes in $G$ whose centralizer in $G$ has order divisible by $p.$  For a fixed $[g] \in K_p$, we will let $\alpha$ denote the highest power of $p$ dividing the order of the centralizer of $g$ in $G$. (This $\alpha$ should not be confused with the $\alpha^\lambda_g(n)$ of \cref{eq:mults}.)
    
Let $h \in G $ be any element in $G$ and let $|h|=n=p^k m$ where $k\geq 0$ and $(p,m)=1.$ Then we can write $h$ as a product $h=ab,$ where $a$ and $b$ commute and $a$ has order $m.$ (Both $a$ and $b$ can be expressed as powers of $h$.) We call $a$ the $p$-\emph{regular part} of $h.$ We note here that if $h'\in G$ is in the same rational conjugacy class as $h,$ then their corresponding $p$-regular parts $a'$ and $a$ are also in the same rational conjugacy class, i.e if $h'\in[h]$ then $a'\in [a]$. 
This allows us to make the following definition.
    \bd 
    For a fixed $p$ and $[g] \in K_p$ as above, the $p$-\emph{regular section} $\Spg_{p,g}$ of $[g]$ is the set of rational conjugacy classes $[h] \in \mathscr{C}_G$ such that the $p$-regular part of $h$ lies in $[g].$
    \ed
For $G=Th, $ and for each prime $p$ dividing $|G|,$ \Cref{tab:cong1} lists the rational conjugacy classes $ [g]$ in $K_p,$ along with their $p$-regular section $\Spg_{p,g}$ and the highest power $\alpha$ such that $p^{\alpha}$ divides the order of the centralizer $C_G(g)$.
    
For a fixed group $G,$ and prime $p$ dividing $|G|$, fix a rational conjugacy class $[g] \in K_p.$ 
Let $\Z_{(p)}=\{\frac{a}{b}\colon a,b\in\Z, p\nmid b \}$ denote the localization of $\Z$ at the prime ideal $(p)$, and let $I:=p^\alpha \Z_{(p)}.$ We define $m:=|\Spg_{p,g}|,$ and let $M_{p,g}$ denote the set of all $m$-tuples  $(l_1,l_2,\dots, l_m)\in \Z^{\oplus m}_{(p)}$ such that  \be\sum^m_{i=1} l_i \chi\left([h]_i\right) \equiv 0 \pmod{I^m}\ee for all irreducible rational characters $\chi$ of $G$ and all rational conjugacy classes $[h]_i$ in $ \Spg_{p,g}.$
    
We are now ready to state the following important lemma. 
    \begin{lemma}\label{LemmaSmith} %
        Assuming the above notation, an integer-valued class function $c : G\to\Z$ of $G$ is a virtual rational character of $G$ if and only if for all primes $p$ and rational conjugacy classes $[g]$, \be\sum^m_{i=1} l_i c\left([h]_i\right) \equiv 0 \pmod{I^m}\ee for all $(l_1,l_2,\dots, l_m)\in M_{p,g}.$
    \end{lemma}
\begin{proof}
    This is a direct application of \citep[Theorem 1.1]{Smith}.
\end{proof}
    
\Cref{LemmaSmith} reduces the problem of checking whether the multiplicities are integral to a $p$-local computation. We illustrate this with an example.
    \begin{example}
    Let $p=19.$ Then, $K_{19}=\{1A\}$ and $\Spg_{19,1A}=\{1A, 19A\}.$ We have $\alpha=1,$ and $M_{19,1A}$ is the set of ordered pairs $(x,y) \in \Z_{(19)}^{\oplus 2}$ such that \be x\chi(1A)+y\chi(19A)\equiv 0 \pmod{19}\ee for each irreducible rational character $\chi$ of the Thompson group. Plugging in values for $\chi(1A)$ and $\chi(19A),$ we find that $M_{19,1A}=\{(x,y) \in \Z_{(19)}^{\oplus 2}~|~ x+y\equiv 0 \pmod{19}\}.$ So in order to prove that $\omega^{\lamda}_{n}:Th \to \C$ is a virtual rational character for each $\lambda \in \Lambda$ and $n \in \Z,$ we need to check that for each $(x,y)\in M_{19, 1A},$  \be x\alpha^\lambda_{1A}(n)+y\alpha^\lambda_{19A}(n)\equiv 0 \pmod{19}\ee where $\Fgt=6q^{-5}+\sum_{n\geq3}\alphg(n)$ (cf. \cref{eq:mults}.) Thus, we have to show that the following congruence is satisfied for every $n \in \N$ and $\lambda\in \Lambda$ \be\label{4.2} \alpha^\lambda_{1A}(n)-\alpha^\lambda_{19A}(n)\equiv 0 \pmod{19}.\ee 
   
    This is a doubly infinite set of congruences (for each fixed $\lamda\in \Lambda$, we have a congruence for every integer $n$), but we can get rid of the dependence on $\lambda$ as follows: Note that $\alpha^\lambda_{1A}(n)$ is independent of $\lambda$ since the cusp form space $S_{1A}$ is empty so we can write $a_{1A}(n)$ for $\alpha^\lambda_{1A}(n)$ (see \Cref{MainProp} for notation). Also by  \Cref{MainProp}, $\alpha^\lambda_{19A}(n)=a_{19A}(n)+(m_{19A}\lamda_{19A}+n_{19A}) b_{19A}(n)$ where $b_{19A}(n)$ is the $n^{th}$ coefficient of $f_{19A}(\tau) \in S_{19A}.$ From \Cref{ngmg}, $m_{19A}=18$ and $n_{19A}=19,$ so checking \cref{4.2} reduces to checking that 
        \be
            a_{1A}(n)=a_{19A}(n)+18b_{19A}(n) \pmod{19}
        \ee 
    for all $n\in N.$
    \end{example}
We can do the same thing for every pair $(p, [g])$ where $[g]\in K_p,$ and get a list of congruences that we need to check in order to show that the function $\omega^{\lamda}_{n}$ is a virtual rational character in every case. An inspection of \Cref{ngmg} and \Cref{fcoeff} confirms that we can always get rid of the dependence on $\lambda$. This still isn't a finite computation because at the moment, we need to check each congruence for all positive integers $n.$ However, that can be easily resolved in the following way: Let $\upsilon(\tau)$ be the unique cusp form in $S^+_{\frac{37}{2}}(\Gamma_0(4))$ whose Fourier expansion is of the form $q^5 - 56q^8 + O(q^9).$ Then for each $[g]$ and $\lambda,$ $\Fgt \upsilon(\tau)$ is a holomorphic modular form of weight 20 and level $|g| h_g$ so Sturm's theorem \citep{Sturm} applies. Thus, it suffices to check that the congruences hold for the first $M$ Fourier coefficients of the holomorphic modular form where $M$ is the Sturm bound which in the worst case is just shy of 4000. As before, we used \citep{PARI2} to check these. 

We conclude this section with another example of this procedure, for clarity.

\begin{example}
    Let $p=3,$ then $K_3=\{1A, 2A, 4A,4B,5A, 7A, 8A, 8B,10A, 13A \}.$ Pick $[g]=1A.$ Then, $\alpha=10,$ $\Spg_{p,g}=\{ 1A, 3A,3B,3C, 9A,9B,9C, 27A, 27BC \}$ and $M_{3,1A}$ is the set of 9-tuples $(y_1, y_2,\dots,y_9)$ in $\Z^{\oplus9}_{(3)}$ such that $$y_1\chi(1A)+y_2\chi(3A)+\dots+y_9\chi(27BC)\equiv0 \pmod{3^{10}} $$
    for each irreducible rational character $\chi$ of the Thompson group.
As before, in order to prove that $\omega^{\lamda}_{n}:Th \to \C$ is a virtual rational character for each $\lambda \in \Lambda$ and $n \in \Z,$ we need to check that for each $(y_1,\dots,y_9)\in M_{3, 1A}$ we have,   
\be y_1\alpha^\lambda_{1A}+y_2\alpha^\lambda_{3A}+\dots+y_9\alpha^\lambda_{27BC}\equiv0 \pmod{3^{10}}.  \ee
This is easier to manage as a matrix computation. We let X denote the $39\times 9$ matrix $$X=[\chi_i(h)]_{0<i\leq39, h\in R_{p,g}},$$ and let $\mathbf{a}:=(a_{1A}, a_{3A}, a_{3B},\dots, a_{27BC}).$ For each $(y_1,\dots,y_9)\in M_{3, 1A},$ we denote by $\mathbf{y}$ the corresponding column vector whose entries are $y_1,y_2,\dots,y_9.$

Note that for all rational conjugacy classes $[g]$ in $\Spg_{3,1A}$ the corresponding cusp form space $S_g$ is empty, so we can in fact reduce to checking that $\mathbf{a}\cdot \mathbf{y}\equiv 0\pmod{3^{10}}$ for all $\mathbf{y}$ such that $X\mathbf{y}\equiv 0 \pmod{3^{10}}.$ In order to check this, we first compute a basis for the $\Z_{(3)}$-span of the row vectors of $X.$ We can use the \texttt{GAP} \citep{GAP} command \texttt{BaseIntMat} 
to do this computation. It turns out that the $\Z_{(3)}$-span of the row vectors of $X$ is the same as that of the row vectors of the following $9\times9$ matrix: 
\[M:=\left[\begin{matrix}      
      1&    1&    1&    1&    1&    1&    1&    1&    1\\
      0&    9& 1944&   72&    0&   45&   24&    0&   15\\
      0&    0& 2187&    0&    0&   27&    0&    3&    3\\
      0&    0&    0&   81&    0&   27&    0&    0&    0\\
      0&    0&    0&    0&   27&   27&    0&    3&    3\\
      0&    0&    0&    0&    0&   81&    9&    6&   15\\
      0&    0&    0&    0&    0&    0&   27&    0&   18\\
      0&    0&    0&    0&    0&    0&    0&    9&    9\\
      0&    0&    0&    0&    0&    0&    0&    0&   27
  \end{matrix}\right].\]

We can solve $M\mathbf{y}\equiv 0 \pmod{3^{10}}$ for $\mathbf{y}$ and then compute $\mathbf{a}.\mathbf{y}$ modulo ${3^{10}}$ to see that the congruences we need to check are: 
    \small{\begin{align}
        \begin{split}
        a_{1A} - a_{3A} & \equiv 0 \pmod{ 3^2} \\
        a_{1A} - a_{9A} & \equiv 0 \pmod{ 3^3} \\
        7 a_{1A} - 8 a_{3A} + a_{3C} & \equiv 0 \pmod{ 3^4}  \\
        215 a_{1A} - 216 a_{3A} + a_{3B} & \equiv 0 \pmod{ 3^7}  \\
        a_{1A} + 27 a_{3A}- a_{3B}- 27 a_{3C} - 81 a_{9A} + 81 a_{9B} & \equiv 0 \pmod{ 3^8}  \\
        1214 a_{1A} - 1971 a_{3A}+ a_{3B} + 27 a_{3C}  + 81 a_{9A}  - 81 a_{9B} + 729 a_{9C}& \equiv 0 \pmod{ 3^9}  \\
        2591 a_{1A}  - 594 a_{3A} + a_{3B}- 54 a_{3C} + 81 a_{9A} + 162 a_{9B} - 2187 a_{27A} &\equiv 0 \pmod{ 3^9}  \\
        1214 a_{1A} + 216 a_{3A} +  a_{3B} + 27 a_{3C}+ 81 a_{9A} - 81 a_{9B} - 1458 a_{9C}+\dots\\- 2187 a_{27A}+2187 a_{27BC}&\equiv 0 \pmod{ 3^{10}}.  
        \end{split}
    \end{align}}

(Alternatively, we can use GAP to check that any given vector $\mathbf{a}$ is in the aforementioned span, by using the following code: \texttt{IsContainedInSpan(MutableBasis(Integers,M),a mod $3^{10}$).}) 

As in Example 1, we can use PARI \citep{PARI2} to check these congruences up to the Sturm bound, which in this case comes out to be less than $1100.$ 

\end{example}

 We can continue in this manner and check that all multiplicities are integral and hence for each $\Fgt$ as described in \Cref{MainProp}, there exists a virtual $Th$-module $W^\lamda$ such that for each $[g]$ in $Th, $ 
 \be\mathcal{F}^{\lambda}_{g}(\tau)=6q^{-5}+\sum_{\substack{n>0\\ n\equiv0,3 \pmod{4}}}\tr\left(g|W^{\lambda}_{n}\right)q^n \ee
 This proves \Cref{thm:main}.
%%%%%%%%%%%%%%%%%%%%%%%%%%%%%%%%%%%%%%%%%%%%%%%%%%%%% ELLIPTIC CURVES %%%%%%%%%%%%%%%%%%%%%%%%%%%%%%%%%%%%
\section{Elliptic Curves}\label{sec:Appl}

The family of $Th$-modules that we get from \Cref{thm:main} encodes arithmetic information about
quadratic twists of elliptic curves with conductors 14 and 19. This is the content of \Cref{Ell19,Ell14}. We will prove these theorems in this section, but first, we have to develop some background. We recall here some basic notation and facts about traces of singular moduli, which were studied by Zagier in \citep{zagier} and have since been examined extensively.

\subsection{Traces of Singular Moduli}\label{tracesofsingularmoduli} %
Let $\calQ_{D}^{(N)}$ be the set of positive definite quadratic forms $Q=[a,b,c]:=ax^2+bxy+cy^2$ of discriminant $-D=b^2-4ac<0$ such that $N|a$. Then, $\Gamma_0(N)$ acts on $\calQ^N_{D}$ with finitely many
orbits. For $Q=[a,b,c]\in\calQ_{D}^{(N)}$, we denote by $\tau_Q:=\frac{-b+i\sqrt{D}}{2a}$ the unique root of $Q(x,1)$ in the upper half-plane $\HH$. Let $f:\HH\rightarrow\C$ be a function invariant under the action of $\Gamma_0(N),$ and $n\equiv 0,3\pmod{4}$ be a positive integer. Then we can define, 
    \begin{equation}\label{eqtrace}
        \Tr_{D}^{(N)}(f;n):=\sum_{Q\in\calQ_{nD}^{(N)}/\Gamma_0(N)} \chi_{D}(Q)\frac{f(\tau_Q)}{\omega^{(N)}(Q)},
    \end{equation}
where $\omega^{(N)}(Q)$ is the order of the stabilizer of $Q$ in $\Gamma_{0}(N)/\{\pm1\}$ and $\chi_D(Q)$ is the genus character for positive definite binary quadratic forms whose discriminants are multiples of $D$, defined as follows (see for example \citep{reuz}): 
    \be
        \chi_D([a,b,c])=\begin{cases}0 &\text{ if } (a,b,c,D)>1\\
        \left(\frac{D}{r}\right)&\text{ if } (a,b,c,D)=1 \text{ and $Q$ represents $r$ with } (r,D)=1.\end{cases}
    \ee
For $N\in \{14, 19\},$ let $J^{(N,+)}$ be the normalized Hauptmodul for the group $\Gamma^{+}_0(N).$ (We know this exists because the corresponding modular curve $X_0^{(+)}(N)$ has genus 0. See \citep{FMN}, or \citep[Table 5.2]{DuncanMertensOno-OnanArxiv}.)
    \begin{proposition}\label{traces}
        Let $N\in \{14, 19\}$ and let $J^{(N,+)}$ as above. Then,
    \be%
        R^{[-5],+}_{\frac 32,4N}(\tau)=q^{-5}+\frac{-2}{3\sqrt{5}}\sum_{\substack{n>0 \\ n\equiv 0,3\smod 4}} \Tr_5^{(N)}(J^{(N,+)};n)q^n
    \ee
    \end{proposition}
\begin{proof}
    This is a direct application of Corollary 1.3 of \citep{reuz}.
\end{proof}
In particular, this means that $f^{wh}_{g}(\tau)$ for $o(g)=N\in\{14, 19\}$ is given by 
    \be
        f^{wh}_{g}(\tau)=6q^{-5}+\sum_{0<n} a_{g}(n) q^n=6q^{-5}-\frac{4}{\sqrt{5}}\sum_{\substack{n>0 \\ n\equiv 0,3\smod 4}} \Tr_5^{(N)}(J^{(N,+)};n)q^n. 
    \ee
Writing $f^{wh}_{g}(\tau)$ in the above form turns out to be essential for the proofs of  \Cref{Ell14,Ell19}. We give here another key lemma which we will use in both proofs.
\begin{lemma}\label{zerotraces}
Let $N\in\{14, 19\}$ and let $d<0$ be a fundamental discriminant that satisfies the respective conditions of \Cref{Ell19,Ell14}; then \be\Tr_5^{({N})}(f,|d|)=0\ee
for any $\Gamma_0 {(N)} $ invariant function $f,$ and hence, in particular, for $f=J^{(N,+)}(\tau).$ 
\end{lemma}
\begin{proof}
    For $N\in\{14, 19\},$ the conditions of the theorems ensure that $5d$ is not a square mod $4N,$ which means that there are no quadratic forms $[a,b,c]$ of discriminant $b^2-4ac=5d$ such that $N|a.$ Thus $\calQ^{(N)}_{|5d|}$ is empty for all such $d$ and thus 
        \be
            \Tr_{5}^{(N)}(f;|d|)=\sum_{Q\in\calQ_{|5d|}^{(N)}/\Gamma_0(N)}     \chi_{5}(Q)\frac{f(\tau_Q)}{\omega^{(N)}(Q)}=0
        \ee
    for any function $f$ that is $\Gamma_0(N)$ invariant.
\end{proof}
We now recall facts about elliptic curves that we will use in order to prove \Cref{Ell19,Ell14}.
\subsection{Elliptic Curves}
To prove our main results, we let $E$ be an elliptic curve over $\Q.$ For $d<0$ a fundamental discriminant, we let $E^d$ denote the $d^{th}$ quadratic twist of $E.$ We let $N$ denote the conductor, $\Omega(E)$ denote the real period and $\Reg(E)$ denote the regulator of $E.$ We refer the reader to standard texts on elliptic curves, e.g. \citep{Silverman1} for the definitions of these invariants.

We let $L_E(s)$ denote the $L$-function associated to $E.$ Then, by the modularity theorem \citep{ModularityTheorem} (see also \citep{WilesModularity1,WilesModularity2}), there exists a unique weight 2 newform $\G_E=\sum^\infty_{n=1}a_E(n)q^n$ of level equal to the conductor of $E$ such that \[L_{E}(s)=\sum^\infty_{n=1}a_E(n)n^{-s},\]
where the right-hand side extends to a holomorphic function on $\C$ \citep{AtkinLehnerHecke}. We let $g_{E}(\z)=\sum^\infty_{n=3}b_E(n)q^n\in S^{+}_{\frac32}(\Gamma_0(4N))$ be the weight $\frac32$ cusp form associated to $\G_E$ under the Shintani lift (see \citep{HofmannLiftings} for an overview of the Shintani lift). For $N\in\{14, 19\}$ the dimension of $S^{+}_{\frac32}(\Gamma_0(4N))$ is 1, so for an elliptic curve of conductor $N,$ the weight $\frac 32$ cusp forms $g_E(\z)$ defined as above are the same as the cusp form $f_g(\z)$ associated to $g\in\{14A,19A\}$ in \Cref{secfunc}. This is the key fact that we employ in order to prove \Cref{Ell14,Ell19}. 

Let $E/\Q$ be an elliptic curve with square-free conductor $N,$ and for each $\ell|N,$ let $\omega_\ell$ denote the eigenvalue of the newform $\G_E\in S_2(\Gamma_0(N))$ associated to $E$ and the Atkin--Lehner involution $W_\ell.$ 

Then we have the following lemma of Duncan, Mertens, and Ono, \citep{DuncanMertensOno-OnanArxiv} (based on results due to Agashe \citep{AgasheBSD} and Kohnen \citep{Kohnen85}, and the generalization of Kohnen's work by Ueda and Yamana \citep{Ueda,UedaYamana}) which connects the $p$-divisibility of the cusp form coefficient to $L_{E^d}(1).$ 
\begin{lemma}\label{kohnen}(see \citep[Lemma 6.5]{DuncanMertensOno-OnanArxiv}) %
    Assume the notation above, and let $p\geq3$ be a prime. Let $d<0$ be a fundamental discriminant satisfying $\left(\frac{d}{\ell}\right)=\omega_\ell$ for each $\ell.$ Denote by $d_0$ the smallest such discriminant. Then we have that
        \be
            \ord_p\left(\frac{L_{E^d}(1)}{\Omega(E^d)}\right)=\ord_p\left(\frac{L_{E^{d_0}}(1)}{\Omega({E^{d_0})}}\right)+\ord_p\left(b_E(|d|)^2\right),
        \ee
    where $E^d$ denotes the $d^{th}$ quadratic twist of $E.$
\end{lemma}
Both our proofs of \Cref{Ell19,Ell14} depend on the above lemma. We are now ready to prove \Cref{Ell19}.

\subsection{Proof of Theorem 1.1}  
Fix $W=W^\lambda$ to be an infinite-dimensional graded $Th$-module that satisfies all the properties listed \Cref{thm:main}. Then, for $g$ an element of order 19 in $Th,$ we can combine \Cref{MainProp} and \Cref{traces} to get the following expression for the coefficients of $\F_{19A}(\z):$ %
    \begin{equation}
        \tr(g|W_n)\equiv \frac{-4}{\sqrt{5}}\Tr_5^{({19})}(J^{({19},+)};n)+(n_{{19A}}+\lamda_{19A} m_{19A})b_{19A}(n),
    \end{equation}
where $b_{19A}(n)$ denotes the $n^{\rm th}$ coefficient of the weight $\frac 32$ cusp form $f_{19A}\in S_{19A}.$
Since $W$ is a virtual module for the Thompson group, we know the following congruence holds for each $p|\#Th$ (and in particular for $p=19$) and for all $n>0$ %(cf. \Cref{sec:I})
    \begin{equation}
        \dim(W_n)\equiv \tr(g_{{p}}|W_n) \pmod{p}.
    \end{equation}
where $g_p$ denotes an element of order $ p.$
Plugging in the values of $n_{g}$ and $m_g$ from \Cref{ngmg}, we get,
    \begin{equation}
        \dim(W_n)\equiv\frac{-4}{\sqrt{5}}\Tr_5^{({19})}(J^{({19},+)};n)+18b_{19A}(n) \pmod{19}.
    \end{equation}
Thus for $n=|d|$ where $d$ is a fundamental discriminant that satisfies the properties of \Cref{Ell19}, we use \Cref{zerotraces} to get:
    \begin{equation}
        \dim(W_{|d|})\equiv \tr(g_{{19}}|W_{|d|})\equiv 18 b_{19A}(|d|)\pmod{19}.
    \end{equation}
This shows that the congruence in the statement of our theorem holds if and only if $19 | b_{19A}(|d|), $ or by \Cref{kohnen}, if and only if 
    \begin{equation}
        \ord_{19}\left(\frac{L_{E^d}(1)}{\Omega(E^d)}\right)>\ord_{19}\left(\frac{L_{E^{d_0}}(1)}{\Omega({E^{d_0})}}\right).
    \end{equation}
A quick \texttt{MAGMA} computation for $d_0=-4$ shows that the right-hand side is 0. Thus, if $\dim(W_{|d|})\not\equiv 0 \pmod{19}, $ then $L_{E^d}(1)\not\equiv 0 \pmod{19},$ and in particular, $L_{E^d}(1)\not=0.$ By Kolyvagin's work \citep{KolyvaginBSD}, this means that $E^d(\Q)$ is finite. This completes the proof of \Cref{Ell19}.
\subsection{Theorem 1.3}
We need to develop some more background before proving \Cref{Ell14}. For $\ell$ prime, we let $c_\ell(E)$ denote the Tamagawa number of $E$ at $\ell,$ defined as the finite index \be c_\ell=[E(\mathbb Q_\ell):E^0(\mathbb Q_\ell)], \ee where $E^0(\mathbb Q_\ell)$ is the subgroup of points which have good reduction at $\ell$. If $E$ has good reduction at $\ell$, then $E(\mathbb Q_\ell)=E^0(\mathbb Q_\ell)$ and $c_\ell=1.$ In particular for a general elliptic curve defined over $\Q,$ we have that $c_\ell=1$ for all but finitely many primes $\ell.$
The following result of C. Skinner (see also \citep{SkinnerUrban}) gives a local version of the {\BSD} for certain elliptic curves.
    \begin{theorem}[\citep{Skinner16}, Theorem C]\label{Skinner}
    Let $E/\Q$ be an elliptic curve and $p\geq 3$ a prime of good ordinary or multiplicative reduction. Assume that the $\Gal(\overline{\Q}/\Q)$-representation $E[p]$ is irreducible and that there exists a prime $p'\neq p$ at which $E$ has multiplicative reduction and $E[p]$ ramifies. If $L_{E}(1)\ne 0$, then we have that
    \be
        \ord_p\left(\frac{L_{E}(1)}{\Omega_E}\right)=\ord_p\left(\#\Sha(E)\prod_\ell c_\ell(E)\right).
    \ee
    If $L_{E}(1)=0$, then we have $\Sel_p(E)\neq\{0\}$.
    \end{theorem}

In order to use \Cref{Skinner} in our proof of \Cref{Ell14}, we first show that each elliptic curve $E$ of conductor 14 satisfies the hypotheses of \Cref{Skinner} in the following lemma.
    \begin{lemma}\label{LemmaBlah}
    Let $d<0$ be a fundamental discriminant for which $\left(\frac{d}{7}\right)=-1$ and $\left(\frac{d}{2}\right)=1$; then for each elliptic curve $E$ of conductor 14 the following are true:
\begin{enumerate}[$(a)$]
\item The $d^{th}$ quadratic twist of $E$ has multiplicative reduction at $p\in\{2,7\};$
\item The $\Gal(\overline{\Q}/\Q)$-representation $E^d[7]$ is irreducible; and 
\item $E^d[7]$ ramifies at $2.$ 
\end{enumerate}

\end{lemma}
\begin{proof}
Let $E/\Q$ be an elliptic curve given by a minimal Weierstrass model
\be
   E: \ \ \  y^2+a_1xy+a_3y=x^{3}+a_2x^{2}+a_4 x+a_6 
\ee
and define the discriminant of $E$ by the equation \be \Delta(E):= -b_2^2b_8-8b_4^3-27b_6^2+9b_2b_4b_6, \ee where $ b_2:=a_1^2+4a_4$, $ b_4:=2a_4+a_1a_3, $ $b_6:=a_3^2+4a_6$ and $b_8:=a^2_1a_6 + 4a_2a_6 - a_1a_3a_4 + a_2a^2_3-a^2_4$. Then $E$ has multiplicative reduction at $p$ if and only if $p$ divides the discriminant of $E$ but not the quantity $c_4(E):=(a_1^2+4a_4)^2-24(2a_4+a_1a_3).$ For each elliptic curve of conductor 14, we have that  $a_1=a_3=1, ~a_2=0$ and $a_3\in \{-2731,-171, -36, -11, -1,4\}$ (see \citep[\href{http://www.lmfdb.org/EllipticCurve/Q/14/a}{Elliptic Curve 14.a}]{LMFDB}).
Thus, for each elliptic curve $E$ of conductor 14, $E$ has multiplicative reduction at $p \in \{2,7\}$. Since twisting by a fundamental discriminant $d$ only changes $\Delta(E)$ and $c_4(E)$ up to a power of $d,$ and $d$ is coprime to 14, this proves part $(a).$

Part $(b)$ follows from a lemma of Serre \citep{Serre} which shows that the Galois representation $E^d[7]$ is surjective and hence irreducible. Finally, part $(c)$ follows from part $(b)$ and (the contrapositive of) Theorem 1.1 of \citep{Ribet}. \end{proof}
We are now ready to prove \Cref{Ell14}.

\newcommand{\g}{g}
\begin{proof}[Proof of Theorem 1.3] Fix  $W=W^\lambda$ to be an infinite-dimensional graded $Th$-module that satisfies all the properties listed \Cref{thm:main}. Let $g$ denote an element of order 14 in $Th.$ As before,  we can combine \Cref{MainProp} and \Cref{traces} to get the following expression for the trace of $\g$ on $W:$
\begin{equation}
 \tr(g|W_n)= \frac{-4}{\sqrt{5}}\Tr_5^{({14})}(J^{({14},+)};n)+(n_{{14A}}+\lamda^{(1)}_{14A} m_{14A})b_{14A}(n).
\end{equation}
Here, $b_{14A}(n)$ denotes the $n^{\rm th}$ coefficient of the weight $\frac 32$ cusp form $f_{14A}\in S_{14A}.$ By \Cref{zerotraces}, we get that for $n=|d|$ where $d$ is a fundamental discriminant that satisfies the properties of \Cref{Ell14}, the first term on the right-hand side of the above equation is 0. Plugging in values of $n^{(1)}_g$ and $m^{(1)}_g$ from \Cref{ngmg}, we get the following congruence
\begin{equation}\label{traceg14}
 \tr(\g|W_{|d|})=\left(42+56\lambda^{(1)}_{14A}\right) b_{14A}(|d|)\pmod{49}.
\end{equation}
Suppose first that $\tr(\g|W_{|d|})\not\equiv0 \pmod{49}.$ Then, $b_{14A}(|d|)\not\equiv 0 \pmod{7}.$ By \Cref{kohnen}, this means that
\be\ord_{7}\left(\frac{L_{E^d}(1)}{\Omega(E^d)}\right)=\ord_{7}\left(\frac{L_{E^{d_0}}(1)}{\Omega({E^{d_0})}}\right).\ee
As before we can use \texttt{MAGMA} to check that the right-hand side of the above equation is 0 for each $E$ of conductor $14.$ Thus, if $\tr(\g|W_{|d|})\not\equiv0 \pmod{49}$ then $\ord_{7}\left(\frac{L_{E^d}(1)}{\Omega(E^d)}\right)=0$ and in particular, $L_{E^d}(1)\not=0.$ By \Cref{LemmaBlah} and \Cref{Skinner}, we have that \be\ord_7\left(\#\Sha(E^d)\prod_\ell c_\ell(E^d)\right)=0.\ee
Thus, $\Sha(E^d)[7]$ is trivial. Furthermore, the Mordell--Weil group $E^d(\Q)$ is finite \citep{KolyvaginBSD}. 

We now consider the case that $\tr(\g|W_{|d|})\equiv0 \pmod{49}$ and assume that $\tr(\g|W_4)\not\equiv 43 \pmod{49}.$ We can once again use \Cref{MainProp} and \Cref{traces} to write
\begin{equation}
 \tr(\g|W_4)= \frac{-4}{\sqrt{5}}\Tr_5^{({14})}(J^{({14},+)};4)+\left(42+56\lamda^{(1)}_{14A} \right)b_{14A}(4)=-6+(42+56\lamda^{(1)}_{14A}).
\end{equation}
Our assumption on $\tr(\g|W_4)$ gives us the congruence $(42+56\lambda^{(1)}_{14A})\not \equiv0 \pmod{49}$ and hence by \Cref{traceg14} we get that $7~|~b_{14A}(|d|).$  By \Cref{kohnen} we get
\be\ord_{7}\left(\frac{L_{E^d}(1)}{\Omega(E^d)}\right)>0.\ee
First suppose that $L_{E^d}(1)=0,$ then $\Sel_p(E)\not=0$ by \Cref{Skinner}. So we can reduce to the case where $L_{E^d}(1)\not=0.$ In that case, again by \Cref{Skinner}, we get,
\be\ord_7\left(\#\Sha(E^d)\prod_\ell c_\ell(E^d)\right)>0.\ee
Thus the only thing left to check is that 7 does not divide any of the Tamagawa numbers $c_\ell(E^d)$ for any choice of $E$ and $d.$ By Theorem VII$.6.1$ in Silverman I \citep{Silverman1}, $c_\ell(E^d)\leq 4$ for most of these cases. The only other possibility is when $E^d$ has split multiplicative reduction at $\ell,$ in which case, $c_\ell(E^d)=\ord_\ell(\Delta(E^d))=\ord_\ell(|d|^6\Delta(E))).$ The conditions on $d$ in the theorem imply that $|d|$ is square-free and coprime to $\Delta(E)$ for all $E$ of conductor 14. Thus, if $ 7 ~| ~c_\ell(E^d), $ for some $\ell$ then $\ell$ lies in $ \{2,7\}$ and $7$ divides $\ord_\ell(\Delta(E))$ which is independent of $d.$ A quick check reveals that this is never the case for an elliptic curve of conductor 14.

\end{proof}
\newpage
\appendix
\section{Tables}
\begin{table}[htp]
    \begin{tabular}{ccccccccccc}
    \\
    \hline
    $[g]$            & 1A       &  2A   & 3A   & 3B   & 3C    & 4A     & 4B     & 5A    &6A  &6B\\
    \hline 
    $v,h$            & 0,1      & 0,1   & 1,3  & 0,1  & 2,3   & 0,1    & 1,2    & 0,1   &1,3 & 2,3\\
    \hline
     & \\
      & \\
    \hline
    $[g]$       &            6C       &  7A   & 8A   & 8B   & 9A& 9B    & 9C   & 10A    &  12AB &12 C  \\
    \hline 
    $v,h$     & 0,1      &0,1    & 1,2       & 1,4      & 0,1   & 0,1   & 1,3 &0,1  & 1,3   & 0,1 \\
    \hline
     & \\
      & \\
    \hline
    $[g]$                  & 12D      &  13A  & 14A  & 15AB  & 18A   & 18B    & 19A    & 20A   &  21A    & 24AB   \\
    \hline 
    $v,h$                & 1,6    & 0,1   & 0,1  & 1,3      & 0,1   & 2,3    & 0,1    & 1,2  & 1,3    & 1,6  \\
    \hline
     & \\
      & \\
    \hline
    $[g]$            & 24CD     &  27A& 27BC    & 28A  & 30AB  & 31AB     & 36A & 36BC   &  39AB        \\
    \hline 
    $v,h$           & 1,12    & 1,3 & 1,3      & 0,1  & 2,3   & 0,1      & 0,1  & 0,1    & 1,3    \\
    \hline
      \\
    \end{tabular}
\caption{Multipliers for each rational conjugacy class.}\label{mult}
\end{table} 
%%%%%%%%%%%%%%%%%%%%%%%%%%%%%%%%%%%%%%%%%%%%%%%%%%%%%%%%%%%%%%%%%%%%%%%%%%%%%%%%%%%%%%%%%%%%%%%%%%%%%%%%%%%%%%%%%%%%%%%%%%%%%%%%%%%%%%%%%%%%%%%%%%%%%%%
\begin{table}[htp]
    \resizebox{\textwidth}{!}{\begin{tabular}{|c||c|c|c|c|c|c|c|c|c|c|c|c|c|c|}\hline
    $p$&2&2&2&2&2&2&2&2&2&3&3&3&3&3\\\hline
    $K_g$&1A&3A&3B&3C&5A&7A&9A&9C&15AB&1A&2A&4A&4B&5A\\\hline
    $\alpha$&15&6&3&4&3&3&4&3&3&10&4&3&1&1\\\hline
    $R_{p,g}$&1A,2A&3A&3B,3C&6C&5A&7A&9A,18A&9C&15AB&1A,3A, 3B&2A,6A&4A&4B&5A\\
    &4A,4B&6B&12C&6A&10A&14A&36A, 36BC&18B&30AB&3C, 9A, 9B&6B, 6C&12AB&12D&15AB\\
    &8A,8B&12AB&24CD&12D&20A&28A&&&&9C,27A,27BC&18A,18B&12C&&\\\hline
    \end{tabular}}\\
    ~\\
    ~\\
    \resizebox{\textwidth}{!}{
    \begin{tabular}{|c||c|c|c|c|c|c|c|c|c|c|c|c|c|c|c|c|c|c|}\hline
    $p$&3&3&3&3&3&5&5&5&5&5&7&7&7&7&13&13&19&31\\\hline 
    $K_g$&7A&8A&8B&10A&13A&1A&2A&3C&4B&6A&1A&2A&3A&4A&1A&3A&1A&1A\\\hline
    $\alpha$&1&1&1&1&1&3&1&1&1&1&2&1&1&1&1&1&1&1\\\hline
    $R_{p,g}$&7A&8A&8B&10A&13A&1A&2A&3C&4B&6A&1A&2A&3A&4A&1A&3A&1A&1A\\
    &21A&24AB&24CD&30AB&39AB&5A&10A&15AB&20A&30AB&7A&14A&21A&28A&13A&39AB&19A&31AB\\ \hline
    
    \end{tabular}}\\
    ~\\
\caption{$p$-regular sections}\label{tab:cong1}
\end{table}
\newpage

{\small %\setlength\lineskiplimit{-6pt}
\begin{longtable}{ c L{12cm}}
    \hline\\
    $ f_{12D}(q)= $ & $ q^{4} - 2q^{8} + 2q^{20} - 2q^{40} - 2q^{52} + 4q^{56} - 2q^{68} + 4q^{88} - q^{100} - 6q^{116}+ 2q^{136} + 4q^{148} + O(q^{150})$ \\
        \\\hline\\ 
    $ f_{14A}(q)= $ & $q^{4} - q^{7} - q^{8} + 2 q^{15} - q^{16} + q^{28} + q^{32} - q^{36} - 2 q^{39} + q^{56} - 2q^{60} + q^{63} + q^{64} - 2 q^{71} + 3 q^{72} + 2 q^{79} - 2 q^{84} - 2 q^{88} + 2 q^{95}- q^{100}-q^{112} + 2 q^{119} + 2 q^{120} - 4 q^{127} - q^{128} - 4 q^{135} + q^{144}+ 4 q^{148} + O(q^{151}) $ \\
      \\\hline\\ 
    $ f_{18B}(q)= $ & $  q^{4} + q^{7} - q^{16} - q^{28} - 3 q^{31} + q^{40} - 2 q^{52} + q^{55} + q^{64} + 2 q^{79} + q^{88} + 2 q^{100} + 2 q^{103} + q^{112}+ 3 q^{124} - q^{127} - 4 q^{136} - 2 q^{148} + O(q^{150})  $ \\
     \\\hline\\ 
    $f_{19A}(q)= $ & $ q^{4} - q^{7} - q^{11} + q^{19} + q^{20} - 2 q^{24} + q^{28} + q^{35} - q^{36} + 2 q^{39} - q^{43} - q^{44} + q^{47} - q^{55} + q^{63}- 2 q^{64} - q^{68} + q^{76} - q^{95} + 3 q^{99} + 2 q^{100} - 2 q^{111} + 2 q^{112} - 2 q^{115} - q^{119} + 2 q^{120} - 2 q^{123}+ q^{131} - 3 q^{139} + q^{140} + O(q^{151})$ \\
      \\\hline\\ 
    $f^{(1)}_{20A}(q)=  $ & $  q^{4} - q^{20} - 2 q^{24} - q^{36} + 2 q^{40} + 2 q^{56} + 2 q^{84} - q^{100} - 2 q^{120} - 4 q^{136} + O(q^{151})$ \\
   \\\hline\\ 
    $ f_{20A}^{(2)}(q)= $ & $ q^{7} - q^{15} - q^{23} + q^{47} + q^{63} - 2 q^{87} + 2 q^{95} + q^{103} - 3 q^{127} - 2 q^{143} + O(q^{151})$ \\
 \\\hline\\ 
    $ f_{21A}^{(1)}(q)= $ & $ q^{4} - q^{11} - q^{16} + q^{23} - q^{28} - q^{32} + q^{35} + q^{44} + q^{56} + q^{64} + 2q^{67} - q^{71} - 2 q^{79} - 2 q^{91}- q^{92} + q^{100} - q^{107} + q^{112} - q^{116} - q^{119} +  2q^{127} - 2 q^{148} + O(q^{150})$ \\
    \\\hline \\
    $ f_{21A}^{(2)}(q)=$ & $ q^{7} + q^{8} - q^{11} - 2q^{16} - q^{23} + q^{32} + 2 q^{43} + q^{56} - q^{71} - 2q^{88}-2 q^{91}- 2 q^{92} + 2 q^{95} + 4 q^{100}+ q^{107} - 2 q^{116} - 4 q^{127} + q^{128} - 2 q^{140} +O(q^{150})$ \\
     \\\hline\\ 
    $ f_{24AB}(q)= $ & $ q^{4} - q^{8} - q^{20} + 2 q^{40} - 2 q^{52} + 2 q^{56} + q^{68} - 4 q^{88} - q^{100} + 3 q^{116} - 2 q^{136} + 4 q^{148} + O(q^{150}) $ \\
    \\\hline \\
    $ f_{24CD}^{(1)}(q)= $ & $ q^{4} - q^{100} + O(q^{150})$ \\
     %\\\hline\\
     \\\hline\\
    $ f_{24CD}^{(2)}(q)=  $ &$ q^{7} - 2 q^{15} - q^{31} + 4 q^{39} - 2 q^{63} - 3 q^{79} + 2 q^{87} + q^{103} + q^{127} - 2 q^{135} + O(q^{150})$
 \\ 
    \\\hline \\
    $ f_{24CD}^{(3)}(q)= $ & $ q^{16} + q^{20} - 2 q^{32} - q^{52} + q^{68} + 2 q^{80} - 3 q^{116} + 2 q^{148} + O(q^{150})$
 \\ 
    
    \\\hline \\
    \\
    \\\hline\\
    $f^{(1)}_{28A}(q)= $ & $ q^{4} - q^{8} - q^{16} + q^{28} + q^{32} - q^{36} + q^{56} - 2 q^{60} + q^{64} + 3 q^{72} - 2 q^{84} - 2 q^{88} - q^{100} - q^{112} + 2 q^{120} - q^{128} + q^{144} + 4 q^{148} + 2 q^{156} - 2 q^{168} - 4 q^{184} + q^{196} - q^{200} + O(q^{201}) $ \\ 
    \\ \hline \\
    $ f^{(2)}_{28A}(q)=$ & $ q^{7} - 2 q^{15} + 2 q^{39} - q^{63} + 2 q^{71} - 2 q^{79} - 2 q^{95} - 2 q^{119} + 4 q^{127} + 4 q^{135} + O(q^{151})$ \\
    \\ \hline \\
    $f^{(1)}_{30AB}(q)= $ & $ q^{4} + 2 q^{15} - q^{16} + 2 q^{24} - 2 q^{36} - 4 q^{39} - q^{40} + 2 q^{55} -2q^{60} + q^{64} - 4 q^{79} + 4 q^{84} - 2 q^{96}- q^{100} + 4 q^{111} - 4 q^{120} + 2 q^{135} + 2q^{136}+ 2 q^{144} + O(q^{150})$ \\
    \\\hline \\
    $f^{(2)}_{30AB}(q)=$ & $ q^{7} + q^{8} + q^{20} - q^{28} - q^{32} - q^{40} - 2 q^{47} + q^{52} - q^{55} - 2q^{68} - q^{80} + q^{88} + 2q^{95} - q^{103}+ q^{112} - q^{127} + q^{128} + 2 q^{143} + q^{148}+O(q^{150})$ \\
    \\\hline \\
  $f^{(3)}_{30AB}(q)=$ & $ q^{11} - q^{15} + q^{16} - q^{19} - q^{20} - q^{24} + q^{35} + q^{36} + 2 q^{39} + q^{44} - q^{55} - q^{56} - 3 q^{59} + q^{60} - q^{76} + 2 q^{79} - q^{80} - 2 q^{84} + q^{91} + q^{96} + 3 q^{104} - 2 q^{111} + 2 q^{115} + 2 q^{120} - q^{131}- q^{135} - 3 q^{136} + q^{140} - q^{144} + O(q^{150})$ \\
  \\\hline \\
    $ f^{(1)}_{31AB}(q)= $ & $ q^{4} - q^{8} - q^{20} - q^{28} + q^{32} + 2q^{35} + q^{36} - 2q^{39} + q^{40} - 2q^{51} + q^{56} -    2q^{59} + 2q^{63} - q^{64} + 2q^{67} + 2q^{71} - q^{72} - q^{76} + 2q^{87} - 2q^{95} -   2q^{103} - 2q^{107} + q^{124} + 2q^{128} - 2q^{132} + q^{140} - 2q^{144} + O(q^{151}),$ \\
      \\\hline \\
       $ f^{(2)}_{31AB}(q)= $ & $  q^{7} - q^{8} - q^{16} + q^{19} - q^{31} + q^{35} + 2q^{36} - q^{40} - 2q^{51} + q^{56} - q^{59       } - q^{63} - q^{64} + q^{71} + q^{72} + 2q^{76} + q^{80} - q^{95} + q^{103} - q^{107} -         2q^{111} - q^{112} + q^{128} + q^{144} + O(q^{151})$ \\

\\\hline \\
$ f^{(1)}_{39AB}(q)= $ & $ q^7-q^{19}+q^{20}-q^{31}-q^{32}-q^{44}-q^{59}+q^{67}+q^{71}+q^{80}+q^{83}+q^{91
   }+q^{104}-q^{119}-2q^{124}+2q^{136}-q^{143}+2q^{148}+O(q^{150})$\\
       \\\hline \\
    $ f^{(2)}_{39AB}(q)= $ & $ q^{8}+q^{15}-q^{19}-q^{24}-q^{28}+q^{31}-q^{39}-q^{44}-q^{47}+q^{52}+q^{72}+q^{76
   }-q^{80}+q^{83}+2q^{84}-q^{96}-q^{99}+q^{112}+2q^{115}+q^{119}+q^{123}-
    q^{124}-q^{135}-q^{136}+O(q^{150})$\\
       \\\hline \\
      
    $ f^{(3)}_{39AB}(q)= $ & $q^{11}-q^{15}-q^{19}+q^{24}-q^{28}+q^{31}-q^{32}+q^{39}+q^{44}+q^{52}-q^{59}-
    q^{71}-q^{72}+q^{76}-q^{80}-2q^{84}+q^{96}+q^{99}+q^{104}+q^{112}+2q^{115} -    q^{119}-q^{123}-q^{124}+q^{128}+q^{135}-q^{136}+q^{143}+O(q^{150})$\\
    \\ \hline \\
    \caption{List of non-zero cusp forms in $S_g$ for each rational conjugacy class $[g]$ of $Th.$}\label{fcoeff} \end{longtable}}
{\begin{table}[htp]
    \begin{tabular}{ccc|ccc|ccc|ccc}
   % &  & & &&&&&\\
    $f^{(i)}_{g}$       & $n^{(i)}_g$ & $m^{(i)}_g$   & $f^{(i)}_{g}$       & $n^{(i)}_g$ & $m^{(i)}_g$   & $f^{(i)}_{g}$       & $n^{(i)}_g$ & $m^{(i)}_g$  & $f^{(i)}_{g}$       & $n^{(i)}_g$ & $m^{(i)}_g$   \\   
%    &  & & &&&&&\\
    \hline
     &  & & &&&&&\\
     $f_{12D}$       & 12 & 24   & $f_{14A}$   & 42    & 56     &  $f_{18B}$ & 0    & 18  &$f_{19A}$      & 18   & 19   \\
    &  & & &&&&& \\
      $f^{(1)}_{20A}$  & 0   &  20   &   $f^{(2)}_{20A}$    &  0  & 20 & $f^{(1)}_{21A}$           &  9  &21   & $f^{(2)}_{21A}$  & 17   &  21    \\
     &  & & &&&&&\\
    $f^{(1)}_{24AB}$    &   0 & 48 & $f^{(1)}_{24CD}$  &   0 & 12 & $f^{(2)}_{24CD}$  &   0 & 12 &  $f^{(3)}_{24CD}$    &   0 & 12 \\
    &  & & &&&&& \\
            $f^{(1)}_{28A}$ &    0 & 14   & 
            $f^{(2)}_{28A}$           &  0   &   28 & $f^{(1)}_{30AB}$    & 3 & 30 & $f^{(2)}_{30AB}$     &   15 &  30   \\  
     &  & & &&&&& \\
    $f^{(3)}_{30AB}$          & 21   &  30 & $f^{(1)}_{31A}$    & 2 & 31   & $f^{(2)}_{31A}$      &    19 & 31 & $f^{(1)}_{39AB}$        & 21   & 39  \\  
        &  & & &&&&&\\
     $f^{(2)}_{39AB}$     & 6 & 39 &$f^{(3)}_{39AB}$     &   6 & 39  &&&\\
      &  & & &&&&&\\
     \hline
    \end{tabular}
\caption{Integers $n^{(i)}_g$ and $m^{(i)}_g$ associated to each cusp form $f^{(i)}_g(\tau)$.}\label{ngmg}
\end{table}}
\newpage
\bibliographystyle{plainnat}
\bibliography{Khaqan--Bibliography}
\end{document}

%% file: PreambleSymbolShorthands.tex
\newcommand{\lamda}{\ensuremath{\lambda}}
\newcommand{\Fg}{\ensuremath{\mathcal{F}_g}}
\newcommand{\F}{\ensuremath{\mathcal{F}}}

\newcommand{\Fgt}{\ensuremath{\mathcal{F}^{\lambda}_{g}(\tau)}}
\newcommand{\Rsums}{Rademacher sums}
\newcommand{\Rsum}{Rademacher sum}
\newcommand{\Fwh}{f^{wh}_g}
\newcommand{\FwhA}{f^{wh}_{21A}}
\newcommand{\FwhB}{f^{wh}_{30AB}}
\newcommand{\Fcf}{f^{(i)}_g}
\newcommand{\alphg}{\ensuremath{\alpha^{\lamda}_g}}
\newcommand{\Th}{\ensuremath{Th}}
\newcommand{\BSD}{Birch and Swinnerton-Dyer Conjecture}
\newcommand{\Reg}{\ensuremath{\operatorname{Reg}}} %regulator of E

\newcommand{\Spg}{R}
\newcommand{\calQ}{\ensuremath{\mathcal{Q}}}
\newcommand{\G}{\ensuremath{\mathscr{G}}}
\newcommand{\Gal}{\ensuremath{\operatorname{Gal}}}

\newcommand{\N}{\ensuremath{\mathbb{N}}}  
\newcommand{\Z}{\ensuremath{\mathbb{Z}}} 
\newcommand\smod[1]{\ \left(\operatorname{mod} #1\right)}
\newcommand{\Q}{\ensuremath{\mathbb{Q}}}
\newcommand{\R}{\ensuremath{\mathbb{R}}}
\newcommand{\C}{\ensuremath{\mathbb{C}}}

\newcommand{\M}{\ensuremath{\mathbb{M}}}
\newcommand{\z}{\ensuremath{\tau}}

\newcommand{\pr}{\ensuremath{\operatorname{pr}}}

\newcommand{\SL}{\operatorname{SL}}

\renewcommand{\Re}{\operatorname{Re}}

\newcommand{\SLZ}{\ensuremath{\SL_2(\Z)}}
\newcommand{\tr}{\operatorname{tr}}
\newcommand{\Sel}{\ensuremath{\operatorname{Sel}}}
\newcommand{\ord}{\ensuremath{\operatorname{ord}}}
    \newcommand{\abcd}{\ensuremath{\left(\begin{smallmatrix} a & b \\ c & d \end{smallmatrix}\right)}}
    \newcommand{\bigabcd}{\ensuremath{\begin{pmatrix} a & b \\ c & d \end{pmatrix}}}
\newcommand{\be}{\begin{equation}}
\newcommand{\ee}{\end{equation}}
\newcommand{\bd}{\begin{definition}}
\newcommand{\ed}{\end{definition}}
\renewcommand{\[}{\be}
\renewcommand{\]}{\ee}
\newcommand{\eps}{\varepsilon}
\newcommand{\HH}{\ensuremath{\mathbb{H}}}

\newcommand{\Tr}{\ensuremath{\operatorname{Tr}}}